\documentclass[12pt]{article}
\usepackage[titletoc,title]{appendix}
\usepackage{amsmath,amsfonts,amssymb,bm,amsthm}
\usepackage[colorlinks,citecolor=blue,linkcolor=blue]{hyperref}
\usepackage{tikz}
\usetikzlibrary{shapes,arrows}
\usetikzlibrary{intersections,shapes.arrows}
\usetikzlibrary{decorations.pathreplacing}
\usetikzlibrary{arrows.meta}
\numberwithin{equation}{section}
\newtheorem{theorem}{Theorem}[section]
\newtheorem{lemma}[theorem]{Lemma}
\newtheorem{proposition}[theorem]{Proposition}
\newtheorem{corollary}[theorem]{Corollary}
\theoremstyle{definition}
\newtheorem{example}[theorem]{Example}
\newtheorem{definition}[theorem]{Definition}

\theoremstyle{remark}
\newtheorem{remark}[theorem]{Remark}

\usepackage[T1]{fontenc}
\usepackage[top=1in, bottom=1in, left=1in, right=1in]{geometry}
\usepackage{fancyhdr}
\usepackage{enumerate}
\usepackage{xcolor,cancel}
\usepackage{relsize}

\renewcommand{\tilde}{\widetilde}

\allowdisplaybreaks

\begin{document}

\title{Invariant subspaces of elliptic systems II:\\ 
spectral theory}
\author{Matteo Capoferri\thanks{MC:
School of Mathematics,
Cardiff University,
Senghennydd Rd,
Cardiff CF24~4AG,
UK;\newline
CapoferriM@cardiff.ac.uk,
\url{http://mcapoferri.com}.
}
\and
Dmitri Vassiliev\thanks{DV:
Department of Mathematics,
University College London,
Gower Street,
London WC1E~6BT,
UK;
D.Vassiliev@ucl.ac.uk,
\url{http://www.ucl.ac.uk/\~ucahdva/}.
}}

\renewcommand\footnotemark{}

\date{
\vspace{1mm}
\it \small In memory of Misha Shubin\\ 
whose support was invaluable to the second author at the beginning of his career}

\maketitle

\vspace{-.6cm}

\begin{abstract}
Consider an elliptic self-adjoint pseudodifferential operator $A$ acting on $m$-columns of half-densities on a closed manifold $M$, whose principal symbol is assumed to have simple eigen\-values.
We show that the spectrum of $A$ decomposes, up to an error with super\-polynomial decay, into $m$ distinct series, each associated with one of the eigen\-values of the principal symbol of $A$. These spectral results are then applied to the study of propagation of singularities in hyperbolic systems. 
The key technical ingredient is the use of the carefully devised pseudodifferential projections introduced in the first part of this work, which decompose $L^2(M)$ into almost-orthogonal almost-invariant subspaces under the action of both $A$ and the hyperbolic evolution.

\


\

{\bf 2020 MSC classes: }
primary 
35P20; 
secondary
47A15,
35J46,
35J47,
35J48,
58J05,
58J40,
58J45. 

\end{abstract}

\tableofcontents

\allowdisplaybreaks

\section{Statement of the problem}
\label{Statement of the problem}

In this paper we continue the analysis of invariant subspaces of elliptic systems initiated in \cite{part1}, focussing on the spectral-theoretic aspects of the problem.

Let $M$ be a connected closed manifold of dimension $d \ge 2$\footnote{When $d=1$ the punctured cotangent bundle $T^*M\setminus\{0\}$ is not connected.  Although this is not a fundamental obstacle, we assume the dimension of the manifold to be at least 2 to avoid repeated discussions on the difference between $d=1$ and $d\ge2$.}.  We denote by $(x^1, \ldots, x^d)$ local coordinates on $M$.

As in \cite{part1}, we denote by $C^\infty(M)$ the linear space of $m$-columns of smooth complex-valued half-densities over $M$ and by $L^2(M)$ its closure with respect to the inner product
\begin{equation*}
\label{inner product on m-columns of half-densities}
\langle v,w\rangle:=\int_M v^*w\,dx\,,
\end{equation*}
where $dx:=dx^1\dots dx^d$.  Accordingly, we denote by $H^s(M)$, $s\in \mathbb{R}$, the corresponding Sobolev spaces. Here and further on $*$ stands for Hermitian conjugation when applied to matrices and for adjunction when applied to operators.

Let $\Psi^s$ be the space of classical pseudodifferential operators of order $s$ acting from $H^s(M)$ to $L^2(M)$. For an operator $B\in \Psi^s$ we denote by $B_\mathrm{prin}$ and by $B_\mathrm{sub}$ its principal and subprincipal symbols, respectively.

Let $A\in\Psi^s$, $s\in\mathbb{R}$, $s>0$, be an elliptic self-adjoint linear operator,
where ellipticity means that
\begin{equation*}
\label{definition of ellipticity}
\det A_\mathrm{prin}(x,\xi)\ne0,\qquad\forall(x,\xi)\in T^*M\setminus\{0\}.
\end{equation*}

Throughout this paper we assume that the eigenvalues of $A_\mathrm{prin}$ are simple. We denote by $m^+$ (resp.~$m^-$) the number of positive (resp.~negative) eigenvalues of
$A_\mathrm{prin}(x,\xi)$.
We denote by $h^{(j)}(x,\xi)$
the eigenvalues of $A_\mathrm{prin}(x,\xi)$ and by $P^{(j)}(x,\xi)$ the corresponding eigenprojections. Eigenvalues are enumerated in increasing order, with positive index $j=1,2,\ldots, m^+$ for positive $h^{(j)}(x,\xi)$ and negative index $j=-1,-2,\ldots, -m^-$ for negative $h^{(j)}(x,\xi)$.

The spectrum of our operator $A:H^s(M)\to L^2(M)$ is discrete and accumulates to infinity.
More precisely,
if $m^+\ge1$ the spectrum accumulates to $+\infty$,
if $m^-\ge1$ the spectrum accumulates to $-\infty$, 
and if $m^+\ge1$ and $m^-\ge1$ the spectrum accumulates to~$\pm\infty$.

Let us recall a few results from \cite{part1} which will be useful later on.

\begin{definition}
\label{definition of sign definiteness modulo}
We say that a symmetric pseudodifferential operator $B$ is \emph{nonnegative (resp.~nonpositive) modulo} $\Psi^{-\infty}$
and write
\[
B\ge0\mod\Psi^{-\infty}\qquad(\text{resp.}\ B\le0\mod\Psi^{-\infty})
\]
if there exists a symmetric operator $C\in\Psi^{-\infty}$
such that $B+C\ge0$ (resp.~$B+C\le0$).
\end{definition}

\begin{theorem}
\label{theorem results from part 1}
Let $A$ be as above and let $\delta$ be the Kronecker symbol.

\begin{enumerate}[(a)]

\item 
\cite[Theorem~2.2]{part1} There exist $m$ pseudodifferential operators $P_j\in \Psi^0$ satisfying
\begin{enumerate}[(i)]
\item 
$(P_j)_\mathrm{prin}=P^{(j)}$,

\item
$P_j=P_j^* \mod \Psi^{-\infty}$,

\item
$P_j P_l =\delta_{jl}\, P_j \mod \Psi^{-\infty}$,

\item
$\sum_j P_j=\mathrm{Id} \mod \Psi^{-\infty}$,

\item
$[A,P_j]=0 \mod \Psi^{-\infty}$.
\end{enumerate}
These operators are uniquely determined, modulo $\Psi^{-\infty}$, by $A$.

\item
\cite[Theorem~2.5]{part1}
We have
\begin{align*}
&P_j^*AP_j\ge0\mod\Psi^{-\infty}\quad\text{for}\quad j=1,\dots,m^+,
\\
&P_j^*AP_j\le0\mod\Psi^{-\infty}\quad\text{for}\quad j=-1,\dots,-m^-.
\end{align*}
\end{enumerate}
\end{theorem}

Theorem~\ref{theorem results from part 1} tells us that, given an elliptic self-adjoint operator $A\in \Psi^s$, one can construct a unique orthonormal basis of pseudodifferential projections commuting with $A$. These projections partition $L^2(M)$ into $m$ invariant subspaces under the action of $A$, modulo $C^\infty(M)$. Furthermore, they allow one to decompose $A$ into precisely $m$ (non-elliptic) sign definite operators $P_j^*AP_j \in \Psi^s$.

In the light of the above results, one would be led to think that the decomposition
\begin{equation*}
\label{decomposition of A in terms of the Pj}
A=\sum_{j}P_j^*AP_j \mod \Psi^{-\infty}
\end{equation*}
could be used, somehow, to obtain a similar decomposition at the level of the spectrum of $A$, at least in the limit $|\lambda| \to +\infty$. It is well-known that, asymptotically, positive eigenvalues of $A_\mathrm{prin}$ account for the positive spectrum of $A$ and negative eigenvalues of $A_\mathrm{prin}$ for the negative spectrum of $A$. One's hope would be to use pseudodifferential projections to achieve a finer partition of the spectrum of $A$ into $m$ distinct families, singling out the contribution of each individual eigenvalue of $A_\mathrm{prin}$.

Now, the naive approach of looking at the spectra of the operators $P_j^*AP_j$ does not look very promising, in that the latter are \emph{not} elliptic, hence the standard spectral-theoretic and asymptotic techniques cannot be applied. If one is to succeed in achieving the above spectral decomposition without abandoning completely the realm of elliptic operators, a more clever strategy is needed.

\begin{remark}
There are at least two other rather natural approaches to the problem at hand. 
\begin{enumerate}
\item The first approach would involve working with negative order operators.  For fixed $\lambda$ in the resolvent set of $A$, one can consider the negative order operators
\[
R_j:=P_j^*(A-\lambda \operatorname{Id})^{-1}P_j, \qquad j=-m^{-},\ldots, -1,1,\ldots,m^+,
\] 
and study the asymptotics of their counting functions as the spectral parameter tends to zero, in the spirit of \cite{birman1, birman2}. 
We decided not to pursue this approach, which presents nontrivial technical obstacles, but to develop a novel strategy instead. The latter will have the advantage of allowing us to obtain our results without the need to work with non-elliptic operators.

\item
The second approach would involve microlocally diagonalizing the operator $A$, i.e.~constructing an almost-unitary operator $B$ such that $B^*AB$ is a diagonal matrix operator, modulo $\Psi^{-\infty}$ --- see, for example, \cite{taylor_diag,cordes,LF91,BR99,NS04,PST03}. In the context of Dirac operators such a $B$ is sometimes referred to as the Foldy--Wouthuysen Transform.
In constructing the operator $B$ one encounters the issue that the almost-unitary operator $B$ is not defined uniquely, not even at the level of the principal symbol,  but only up to gauge transformations (see also \cite[Sec.~5]{CDV}).  
Neglecting to account for these gauge transformations and the curvatures that they bring about has led to mistakes in some publications, see \cite[Sec.~11]{CDV}. Furthermore, there may be topological obstructions to the existence of a global diagonalization: the issue here is that it is not always possible to adjust the gauge so that the eigenvectors of the principal symbol become (globally defined) smooth $m$-columns on $T^*M\setminus\{0\}$. Finally, the diagonalization procedure affects the asymptotics of the local counting function,  in that it involves conjugation by a pseudodifferential operator. This makes diagonalization, in a sense, a less natural approach when it comes to studying the spectral properties of $A$.  We refer the reader to \cite{diagonalization} for a detailed analysis of the almost-unitary operator $B$.
The use of pseudodifferential projections has the advantage of circumventing these issues altogether. 
\end{enumerate}

\end{remark}

All in all, the results from Theorem~\ref{theorem results from part 1} warrant the following natural questions.

\

\textbf{Question 1\ }
Can we exploit the pseudodifferential projections $P_j$ to achieve a partition of the spectrum of $A$ into $m$ disjoint families of eigenvalues?

\

\textbf{Question 2\ }
Can we exploit the pseudodifferential projections $P_j$ to advance the current understanding of spectral asymptotics for elliptic systems?

\

\textbf{Question 3\ }
Can we exploit the pseudodifferential projections $P_j$ to advance the current understanding of propagation of singularities for hyperbolic systems?

\

The goal of this paper is to provide a rigorous affirmative answer to Questions 1, 2 and~3.

\section{Main results}
\label{Main results}

Our main results can be summarised in the form of five theorems stated in this section.

We will assume that $m^+\ge1$ and will be dealing with the asymptotics of
the positive eigenvalues of $A$. The case of negative eigenvalues can be handled by replacing
$A$ with $-A$.

Let
\begin{equation}
\label{positive eigenvalues of A}
0<\lambda_1\le\lambda_2\le\dots\le\lambda_k\le\dots\to+\infty
\end{equation}
be the positive eigenvalues of $A$ enumerated in increasing order with account of multiplicity.
The task at hand is to partition the eigenvalues \eqref{positive eigenvalues of A}
into $m^+$ separate series corresponding to the $m^+$ different positive eigenvalues
of $A_\mathrm{prin}$. Hence, we will assume that $m^+\ge2$.

In order to partition eigenvalues \eqref{positive eigenvalues of A} into series we introduce the operators
\begin{equation}
\label{operators Aj}
A_j:=A-2\underset{l\ne j}{\underset{l=1,\dots,m^+}\sum}P_l^*AP_l\,,
\qquad j=1,\dots,m^+.
\end{equation}
Each operator $A_j$ is `simpler' than our original operator $A$ in that the principal symbol of~$A_j$
\begin{equation}
\label{Aj prin}
(A_j)_\mathrm{prin}=h^{(j)}P^{(j)}-\sum_{l\ne j} |h^{(l)}| P^{(l)}
\end{equation}
has only one positive eigenvalue, namely, $h^{(j)}(x,\xi)$.
Note also that formula \eqref{operators Aj} implies
\begin{equation}
\label{Aj almost commute}
[A_j,A_l]=0\mod\Psi^{-\infty},\qquad j,l=1,\dots,m^+.
\end{equation}

Let $\theta$ be the Heaviside function.
For a self-adjoint operator $B$ we denote by
\begin{equation}
\label{nonnegative  part of B}
B^+:=B\,\theta(B)=\frac12(B+|B|)
\end{equation}
its nonnegative part.
Then \cite[Theorem 2.7]{part1} implies
\begin{equation}
\label{nonnegative part of A equation 1}
A^+=\sum_{j=1}^{m^+}A_j^+\mod\Psi^{-\infty},
\end{equation}
\begin{equation}
\label{nonnegative part of A equation 2}
A_j^+A_l^+=0\mod\Psi^{-\infty},\qquad j,l=1,\dots,m^+,\quad j\ne l.
\end{equation}
Examination of formulae
\eqref{Aj almost commute},
\eqref{nonnegative part of A equation 1}
and
\eqref{nonnegative part of A equation 2}
suggests that there should be a relation between the positive spectra of $A$ and $A_j\,$,
see Appendix~\ref{appendix on linear algebra},
and Remark~\ref{compelling remark} therein in particular,
for a compelling argument to this effect.
The precise nature of this relation is established by the three theorems given below.

For a self-adjoint operator $B$ we denote its spectrum by $\sigma(B)$,
and we denote by
\begin{equation}
\label{positive part of B spectrum}
\sigma^+(B):=\sigma(B)\cap(0,+\infty)=\sigma(B^+)\setminus\{0\}
\end{equation}
its positive part.
Let
\begin{equation}
\label{positive eigenvalues of Aj}
0<\lambda^{(j)}_1\le\lambda^{(j)}_2\le\dots\le\lambda^{(j)}_k\le\dots\to+\infty
\end{equation}
be the positive eigenvalues of $A_j$ enumerated in increasing order with account of multiplicity.
The following two theorems show that the positive eigenvalues of the operators $A_j$, $j=1,\dots,m^+$,
approximate the positive eigenvalues of the operator $A$, and vice versa.

\begin{theorem}
\label{Main results theorem 4}
For each $j=1,\dots,m^+$ we have
\begin{equation}
\label{Main results theorem 4 equation 1}
\operatorname{dist}\bigl(\lambda^{(j)}_k,\sigma^+(A)\bigr)=O(k^{-\infty})
\quad\text{as}\quad k\to+\infty.
\end{equation}
\end{theorem}

\begin{theorem}
\label{Main results theorem 5}
We have
\begin{equation}
\label{Main results theorem 5 equation 1}
\textstyle
\operatorname{dist}\bigl(\lambda_k\,,\bigcup_{j=1}^{m^+}\sigma^+(A_j)\bigr)=O(k^{-\infty})
\quad\text{as}\quad k\to+\infty.
\end{equation}
\end{theorem}

Theorems \ref{Main results theorem 4} and \ref{Main results theorem 5}
do not quite achieve the sought after partition of the spectrum
\eqref{positive eigenvalues of A} in that they do not establish a one-to-one
correspondence between the positive eigenvalues of the operator $A$
and the positive eigenvalues of the operators $A_j\,$, $j=1,\dots,m^+$.
The issue here is that formulae
\eqref{Main results theorem 4 equation 1}
and
\eqref{Main results theorem 5 equation 1}
establish asymptotic closeness of the spectra
but do not provide sufficient information on the closeness of individual eigenvalues
enumerated in our particular way.
The following theorem addresses this issue and shows that
the above construction is indeed `precise'. 

Let us combine the sequences \eqref{positive eigenvalues of Aj}, $j=1,\dots,m^+$, into one
sequence and denote it by
\begin{equation}
\label{positive eigenvalues of Aj combined}
0<\mu_1\le\mu_2\le\dots\le\mu_k\le\dots\to+\infty.
\end{equation}
Here we combine them with account of multiplicities.

\begin{theorem}
\label{Main results theorem 6}
For any $\alpha>0$
there exists an $r_\alpha\in\mathbb{Z}$ such that
\begin{equation}
\label{Main results theorem 6 equation 1}
\lambda_k=\mu_{k+r_\alpha}+O(k^{-\alpha})
\quad\text{as}\quad k\to+\infty.
\end{equation}
\end{theorem}

Theorem~\ref{Main results theorem 6} will allow us to derive two-term asymptotic
formulae for the eigenvalue counting function of $A$
refining previous results \cite{CDV,AFV,ASV}, see Section~\ref{Refined spectral asymptotics}.
Of course, Theorems~\ref{Main results theorem 4} and~\ref{Main results theorem 5}
follow from Theorem~\ref{Main results theorem 6}, but we listed them as separate results
for the sake of logical clarity.

Our last major result is an application of the above technology to first order hyperbolic systems.
Let $A$ be first order, $s=1$, without any restrictions on $m^+$ and $m^-$.
Consider the associated hyperbolic initial value problem
\begin{equation}
\label{hyperbolic equation 1}
\left(-i\frac{\partial}{\partial t}+A\right)v=0,
\qquad
\left.v\right|_{t=0}=v_0.
\end{equation}
We call \emph{propagator} the solution operator of \eqref{hyperbolic equation 1}, namely,
the time-dependent unitary operator
\begin{equation}
\label{definition of propagator}
U(t):=e^{-itA}.
\end{equation}

It was shown in \cite{CDV,dirac} that $U(t)$ can be approximated,
modulo $C^\infty(\mathbb{R};\Psi^{-\infty})$
(i.e.~modulo an integral operator with infinitely smooth time-dependent integral kernel),
by the sum of precisely $m$ invariantly defined oscillatory integrals $U^{(j)}(t)$ global in space and in time.
Each oscillatory integral $U^{(j)}(t)$ is a Fourier integral operator whose Schwartz kernel
is a Lagrangian distribution associated with the Lagrangian
submanifold of $T^*\mathbb{R}\times T^*M\times T^*M$
generated by the Hamiltonian flow of $h^{(j)}$.
These are $m$ distinct smooth manifolds which encode information on the propagation
of singularities in the hyperbolic system \eqref{hyperbolic equation 1}.

\begin{theorem}
\label{Main results theorem 7}
Let $A\in\Psi^1$ be an elliptic self-adjoint first order $m\times m$ operator.
Suppose that the eigenvalues of its principal symbol are simple.
Then
\begin{equation}
\label{Main results theorem 7 equation 1}
U^{(j)}(t)=P_j\,U(t)=U(t)\,P_j\mod C^\infty(\mathbb{R};\Psi^{-\infty}),
\quad j\in
\{-m^-,\dots,-1,1,\dots,m^+\}.
\end{equation}
\end{theorem}

In fact, we will prove a stronger result, see~Corollary~\ref{corollary commuation of U(j) and P_j}.

Another important special case is that of nonnegative second order operators.
For example, the operator of linear elasticity (Lam\'e operator) falls into this category,
see \cite[Subsection~8.2]{part1} for details.
For such operators we have $m=m^+$ and the propagator is defined as
\begin{equation}
\label{definition of propagator second order}
U(t):=e^{-it\sqrt{A}}\,.
\end{equation}
By means of a suitable modification of techniques from  \cite{CDV,dirac}
it will be shown in Section~\ref{Invariant subspaces in hyperbolic systems}
that in this case the propagator can also be approximated,
modulo $C^\infty(\mathbb{R};\Psi^{-\infty})$,
by the sum of precisely $m$ invariantly defined oscillatory integrals $U^{(j)}(t)$ global in space and in time.
This leads to the following analogue of Theorem~\ref{Main results theorem 7}.

\begin{theorem}
\label{Main results theorem 8}
Let $A\in\Psi^2$ be a nonnegative elliptic self-adjoint second order $m\times m$ operator.
Suppose that the eigenvalues of its principal symbol are simple.
Then
\begin{equation}
\label{Main results theorem 8 equation 1}
U^{(j)}(t)=P_j\,U(t)=U(t)\,P_j\mod C^\infty(\mathbb{R};\Psi^{-\infty}),
\qquad j\in
\{1,\dots,m\}.
\end{equation}
\end{theorem}

\begin{remark}
Theorem~\ref{Main results theorem 8} admits a further generalisation to the case when $A$ is a nonnegative operator of positive even order $2n$. In this case, the propagator is defined as 
\begin{equation}
\label{propagator 2n}
U(t)=e^{-itA^{1/2n}},
\end{equation}
compare with \eqref{definition of propagator second order}. Proving Theorem~\eqref{Main results theorem 8} in this more general case does not present any additional difficulties: one can retrace the arguments given in subsection~\ref{Nonnegative second order operators} replacing $\sqrt{A}$ with $A^{1/2n}$, as appropriate. In particular, as explained in subsection~\ref{Nonnegative second order operators}, one does not need to actually compute $A^{1/2n}$ in order to construct the operator \eqref{propagator 2n}.
The reason why we state our main result for second order operators is twofold. On the one hand it simplifies the presentation, and on the other hand the case $n=1$ has a clearer physical meaning.
\end{remark}

Note that Theorems \ref{Main results theorem 7} and \ref{Main results theorem 8}
cannot be obtained by elementary functional-analytic
arguments involving an expansion over eigenvalues and eigenfunctions of the operator $A$.
Formulae \eqref{Main results theorem 7 equation 1}
and
\eqref{Main results theorem 8 equation 1}
are to do with the propagation of singularities,
a phenomenon which is not detected by the Spectral Theorem.

\

The paper is structured as follows.

Section~\ref{Spectral analysis: partitioning the spectrum} is the core of our paper: it contains the proofs of Theorems~\ref{Main results theorem 4}--\ref{Main results theorem 6}. In subsection~\ref{Separating positive eigenvalues into $m^+$ distinct series} we show that the positive spectrum of $A$ is approximated by the union of the spectra of the $A_j$, $j=1,\ldots, m^+$, and \emph{vice versa}, up to an error of order $O(\lambda^{-\infty})$. In subsection~\ref{Spectral completeness} we demonstrate that our construction is asymptotically precise, namely, that when performing the above approximation no eigenvalue is missed. A key ingredient is a carefully devised partition of the positive semi-axis, provided in subsubsection~\ref{Partition of the positive semi-axis}.

Section~\ref{Invariant subspaces in hyperbolic systems} is concerned with the analysis of hyperbolic systems. In subsection~\ref{First order operators} we focus on first order systems: after briefly recalling the construction of the propagator (wave group), we analyse the relation between the representation of the latter in terms of oscillatory integrals and our pseudodifferential projections, thus proving Theorem~\ref{Main results theorem 7}. In subsection~\ref{Nonnegative second order operators} we perform a similar analysis for nonnegative second order operators, proving Theorem~\ref{Main results theorem 8}.

Finally, in Section~\ref{Refined spectral asymptotics} we show how results from Section~\ref{Spectral analysis: partitioning the spectrum}, Section~\ref{Invariant subspaces in hyperbolic systems} and the first part of this work \cite{part1} can be used to refine our understanding of asymptotic distribution of eigenvalues for first order systems.

The paper is complemented by two appendices.

\section{Spectral analysis: partitioning the spectrum}
\label{Spectral analysis: partitioning the spectrum}

\subsection{Separating positive eigenvalues into $m^+$ distinct series}
\label{Separating positive eigenvalues into $m^+$ distinct series}

The goal of this subsection is to show that the positive spectrum of the operators $A_j$ defined by \eqref{operators Aj} and the positive spectrum of $A$ are mutually close, so as to prove Theorems~\ref{Main results theorem 4} and~\ref{Main results theorem 5}.

Further on in this section, all estimates are to be understood as asymptotic estimates as $k\to+\infty$, unless otherwise specified.

\begin{proof}[Proof of Theorem~\ref{Main results theorem 4}]
Let $u^{(j)}_k$ be a normalised eigenfunction of $A_j$ corresponding to the eigenvalue $\lambda^{(j)}_k$, i.e.
\begin{equation}
\label{Proof Main results theorem 4 equation 1}
A_j u^{(j)}_k=\lambda^{(j)}_k \,u^{(j)}_k,
\end{equation}
\begin{equation}
\label{Proof Main results theorem 4 equation 2}
\|u^{(j)}_k\|_{L^2}=1.
\end{equation}
For every $S\in \Psi^{-\infty}$,
in view of the identity 
\[
Su^{(j)}_k=(\lambda^{(j)}_k)^{-n} S(A_j)^n\, u^{(j)}_k,
 \qquad n=1,2,\ldots,
 \]
and Weyl's law
\begin{equation}
 \label{Proof Main results theorem 4 equation 3}
\lambda^{(j)}_k=\left(\frac1{(2\pi)^d}\underset{h^{(j)}(x,\xi)<1}{\int} \operatorname{dVol}_{T^*M} \right)^{-s/d} k^{s/d}
+o(k^{s/d}),
\end{equation}
see Theorem~\ref{Weyl asymptotics for elliptic systems theorem},
we have
\begin{equation}
\label{Proof Main results theorem 4 equation 6}
Su^{(j)}_k=O(k^{-\infty}).
\end{equation}
The above asymptotic estimate
(as well as similar estimates in subsequent formulae)
is understood in the strongest possible sense: any given partial derivative is estimated by any given negative power of $k$ uniformly over $M$.

We claim that 
\begin{equation}
\label{Proof Main results theorem 4 equation 9}
\|P_l u^{(j)}_k\|_{L^2}=O(k^{-\infty})\quad \text{for}\quad l\ne j\,.
\end{equation}
Indeed, taking into account \eqref{operators Aj} and using \eqref{Proof Main results theorem 4 equation 6}, for $l\ne j$ we have
\begin{equation*}
\label{Proof Main results theorem 4 equation 10}
-\operatorname{sgn}(h^{(l)})\,P_l^*AP_l u^{(j)}_k=\lambda^{(j)}_k P_lu^{(j)}_k+O(k^{-\infty}),
\end{equation*}
which implies
\begin{equation}
\label{Proof Main results theorem 4 equation 11}
-\operatorname{sgn}(h^{(l)})\,
\langle P_l u^{(j)}_k,P_l^*AP_l u^{(j)}_k \rangle
=
\lambda^{(j)}_k\|P_l u^{(j)}_k \|_{L^2}^2+O(k^{-\infty}).
\end{equation}
Combining \eqref{Proof Main results theorem 4 equation 11} with Theorem~\ref{theorem results from part 1}(b) and using once again \eqref{Proof Main results theorem 4 equation 6}, we obtain
\begin{equation*}
\label{Proof Main results theorem 4 equation 12}
\lambda^{(j)}_k\|P_l u^{(j)}_k \|_{L^2}^2\le O(k^{-\infty}),
\end{equation*}
which is equivalent to \eqref{Proof Main results theorem 4 equation 9}.

Formulae
\eqref{Proof Main results theorem 4 equation 1}--\eqref{Proof Main results theorem 4 equation 6}
imply
\begin{equation*}
\label{Proof Main results theorem 4 equation 14bis}
A_j P_l u^{(j)}_k=\lambda^{(j)}_k P_l u^{(j)}_k+O(k^{-\infty}),
\end{equation*}
which combined with 
\eqref{Proof Main results theorem 4 equation 3}
and
\eqref{Proof Main results theorem 4 equation 9}
yields
\begin{equation}
\label{dima1}
\|A_jP_l u^{(j)}_k\|_{L^2}=O(k^{-\infty}) \quad \text{for}\quad l\ne j.
\end{equation}
By elliptic regularity formulae
\eqref{Proof Main results theorem 4 equation 9}
and
 \eqref{dima1} give us
\begin{equation}
\label{Proof Main results theorem 4 equation 15}
P_l u^{(j)}_k=O(k^{-\infty}) \quad \text{for}\quad l\ne j.
\end{equation}
 
Now, \eqref{operators Aj} and \eqref{Proof Main results theorem 4 equation 1} imply
\begin{equation*}
\label{Proof Main results theorem 4 equation 14a}
Au^{(j)}_k=\lambda^{(j)}_ku^{(j)}_k+2\underset{l\ne j}{\underset{l=1,\dots,m^+}\sum}P_l^*AP_l u^{(j)}_k,
\end{equation*}
which, on account of \eqref{Proof Main results theorem 4 equation 15}, can be recast as
\begin{equation}
\label{Proof Main results theorem 4 equation 14}
Au^{(j)}_k=\lambda^{(j)}_ku^{(j)}_k +O(k^{-\infty}).
\end{equation}
Formulae \eqref{Proof Main results theorem 4 equation 2} and \eqref{Proof Main results theorem 4 equation 14} yield \eqref{Main results theorem 4 equation 1}.
\end{proof}

\begin{proof}[Proof of Theorem~\ref{Main results theorem 5}]
Let $u_k\in L^2(M)$ be a normalised eigenfunction of $A$ corresponding to the eigenvalue $\lambda_k>0$, i.e.
\begin{equation}
\label{Proof Main results theorem 5 equation 1}
A u_k=\lambda_k u_k,
\end{equation}
\begin{equation}
\label{Proof Main results theorem 5 equation 2}
\|u_k\|_{L^2}=1.
\end{equation}
The task at hand is to show that there exists a $j\in\{1,\ldots,m^+\}$ such that
\begin{equation}
\label{Proof Main results theorem 5 equation 3}
A_j v^{(j)}=\lambda_k v^{(j)} +O(k^{-\infty})
\end{equation}
for some smooth $v^{(j)}$ with $\|v^{(j)}\|_{L^2}=1$.
Indeed, formula \eqref{Proof Main results theorem 5 equation 3} and the fact that $\|v^{(j)}\|_{L^2}=1$
imply \eqref{Main results theorem 5 equation 1}.

Arguing as in the proof of Theorem~\ref{Main results theorem 4}, one can show that for every $S\in \Psi^{-\infty}$ we have
\begin{equation}
\label{Proof Main results theorem 5 equation 4}
Su_k=O(k^{-\infty}),
\end{equation}
where the asymptotic estimate
(as well as similar estimates in subsequent formulae)
is understood in the strongest possible sense: any given partial derivative is estimated by any given negative power of $k$ uniformly over $M$. 

We claim that
\begin{equation}
\label{Proof Main results theorem 5 equation 5}
P_l u_k=O(k^{-\infty})\qquad\text{for every } l \in\{-1, \ldots, -m^-\}.
\end{equation}
Indeed, formula \eqref{Proof Main results theorem 5 equation 1}, Theorem~\ref{theorem results from part 1}(a) and formula \eqref{Proof Main results theorem 5 equation 4} imply
\begin{equation}
\label{Proof Main results theorem 5 equation 6}
P_l^* A P_l u_k=\lambda_k P_l u_k +O(k^{-\infty})
\end{equation}
for every $l$,
which, in turn, implies
\begin{equation}
\label{Proof Main results theorem 5 equation 7}
\langle u_k, P_l^* A P_l u_k\rangle =\lambda_k \|P_l u_k\|_{L^2}^2+O(k^{-\infty}).
\end{equation}
 Combining \eqref{Proof Main results theorem 5 equation 7} with Theorem~\ref{theorem results from part 1}(b) and using once again \eqref{Proof Main results theorem 5 equation 4} we obtain, for $l<0$,
\begin{equation*}
\label{Proof Main results theorem 5 equation 8}
\lambda_k \|P_lu_k\|_{L^2}^2\le O(k^{-\infty})
\end{equation*}
and hence
\begin{equation}
\label{Proof Main results theorem 5 equation 9}
 \|P_lu_k\|_{L^2}=O(k^{-\infty}).
\end{equation}
By elliptic regularity, \eqref{Proof Main results theorem 5 equation 6},
\eqref{Proof Main results theorem 5 equation 9}
and
\eqref{Proof Main results theorem 5 equation 2}
give us \eqref{Proof Main results theorem 5 equation 5}
(recall that $P_l$ and $A$ commute modulo $\Psi^{-\infty}$).

Now, in view of properties of differential projections, \eqref{Proof Main results theorem 5 equation 4} 
and \eqref{Proof Main results theorem 5 equation 5}, we have
\begin{equation}
\label{Proof Main results theorem 5 equation 11}
u_k=\sum_{l=1}^{m^+} P_l u_k +O(k^{-\infty}).
\end{equation}
Formulae \eqref{Proof Main results theorem 5 equation 11} and \eqref{Proof Main results theorem 5 equation 2} imply that
\begin{equation}
\label{Proof Main results theorem 5 equation 12}
\|P_j u_k\|_{L^2}\ge \frac{1}{m^++1}
\end{equation}
for some $j\in\{1,\ldots, m^+\}$.
By direct inspection we have
\begin{equation}
\label{Proof Main results theorem 5 equation 13}
A_j P_j u_k=\lambda_k P_j u_k +O(k^{-\infty}),
\end{equation}
see \eqref{operators Aj},
\eqref{Proof Main results theorem 5 equation 1}
and \eqref{Proof Main results theorem 5 equation 4}. 

Formulae \eqref{Proof Main results theorem 5 equation 12} and \eqref{Proof Main results theorem 5 equation 13} give us \eqref{Proof Main results theorem 5 equation 3} with $v^{(j)}=P_j u_k/\|P_j u_k\|_{L^2}$.
\end{proof}

We summarise below in the form of a proposition some of the results obtained along the way in the above proofs, as they will be useful later on.

\begin{proposition}
\label{proposition on action of projections on eigenfunctions}
\
\begin{enumerate}[(a)]
\item Let $u_k$ be a normalised eigenfunction of $A$ corresponding to the eigenvalue $\lambda_k$. Then
\begin{equation}
\label{proposition on action of projections on eigenfunctions equation 1}
P_l u_k=O(k^{-\infty})\quad\text{for every}\quad l \in\{-1, \ldots, -m^-\}.
\end{equation}

\item
Let $u_k^{(j)}$ be a normalised eigenfunction of $A_j$ corresponding to the eigenvalue $\lambda_k^{(j)}$.
Then
\begin{equation}
\label{proposition on action of projections on eigenfunctions equation 2}
P_l u_k^{(j)}=O(k^{-\infty})\quad \text{for every}\quad l\ne j.
\end{equation}

\item
Under the same assumptions of part (b) we have
\begin{equation}
\label{proposition on action of projections on eigenfunctions equation 3}
Au_k^{(j)}=\lambda_k^{(j)}u_k^{(j)}+O(k^{-\infty}).
\end{equation}
\end{enumerate}
\end{proposition}

\subsection{Spectral completeness}
\label{Spectral completeness}

The goal of this subsection is to prove Theorem~\ref{Main results theorem 6}.
The first steps in this direction were Theorems~\ref{Main results theorem 4} and~\ref{Main results theorem 5}
which we proved in the previous subsection.
The missing ingredient is taking account of the enumeration of eigenvalues,
i.e.~showing that none were missed when approximating the positive spectrum of $A$
by the positive spectra of the $A_j$, $j=1,\dots,m^+$.

\begin{remark}
Note that if \eqref{Main results theorem 6 equation 1} from Theorem~\ref{Main results theorem 6} holds for some $\alpha>0$, then it holds for all $0<\widetilde\alpha\le\alpha$ with $r_{\widetilde\alpha}=r_\alpha$. We will make use of this fact at various points of forthcoming arguments:
whenever required, we will assume, without loss of generality, that $\alpha$ is as large as needed.
\end{remark}

The proof of Theorem~\ref{Main results theorem 6} is more sophisticated
than that of Theorems~\ref{Main results theorem 4} and~\ref{Main results theorem 5}.
It requires devising a carefully chosen partition of the positive semi-axis and a number of preparatory
results which will be given in subsubsections~\ref{Partition of the positive semi-axis}
 and~\ref{Preparatory lemmata} respectively, before addressing the actual
proof in subsubsection~\ref{Proof of Theorem 2.3}.

Throughout this section we adopt the following notation:
\begin{equation*}
N(\lambda;\rho):=\#\{\lambda_k\ |\ \lambda-\rho\le\lambda_k\le \lambda+\rho\},
\end{equation*}
\begin{equation*}
\widetilde{N}(\lambda;\rho):=\#\{\mu_k\ |\ \lambda-\rho\le\mu_k\le \lambda+\rho\}.
\end{equation*}
We will use the capital letter $C$ for denoting some positive constants,
the precise values of which are unimportant and may change
from line to line.

\subsubsection{Partition of the positive semi-axis}
\label{Partition of the positive semi-axis}

We seek a partition of the positive semi-axis $(0,+\infty)$ into subintervals
$(\nu_n,\nu_{n+1}]$, $n=0,1,2,\dots$, satisfying the following properties:
\begin{enumerate}[(a)]
\item
$\lim_{n\to\infty}\nu_n=+\infty$,
\item
the length of these intervals, $\nu_{n+1}-\nu_n$, tends to zero in such a way that it would allow us to
achieve the required remainder term estimate in \eqref{Main results theorem 6 equation 1},
\item
each $\nu_n$, $n=1,2,\dots$, is at a distance $\gtrsim n^{-\gamma}$ from the set of all eigenvalues \eqref{positive eigenvalues of A}
and \eqref{positive eigenvalues of Aj combined}, for some sufficiently large $\gamma>1$.
\end{enumerate}

Let $\alpha>0$ be the exponent from Theorem~\ref{Main results theorem 6}. Put
\begin{equation}
\label{definition of nu_0}
\nu_0:=0,
\end{equation}
\begin{equation}
\label{definition of nu_n}
\nu_n:=n^\beta+c_n \,n^{-1}, \qquad n=1,2,\ldots,
\end{equation}
where 
\begin{equation}
\label{definition of beta}
\beta:=\frac{1}{1+\frac{\alpha d}{s}}
\end{equation}
and the $c_n$ are some real numbers. Note that $\beta\in (0,1)$ and it can be made arbitrarily small by choosing $\alpha$ sufficiently large.

\begin{lemma}
\label{lemma monotone increasing}
If
\begin{equation}
\label{lemma monotone increasing equation 1}
c_n\in [-\beta/4,\beta/4] \quad \text{for all} \quad n=1,2,\ldots,
\end{equation}
then the $\nu_n$ form a strictly increasing sequence.
\end{lemma}

\begin{proof}
Our choice of $\beta$ guarantees $\nu_0<\nu_1$. Therefore
it is enough to show that
\begin{equation}
\label{proof lemma monotone increasing 1}
n^\beta+\frac{\beta}{4n}<(n+1)^\beta-\frac{\beta}{4n}
\end{equation}
 for all $n=1,2,\ldots$.
 Clearly, \eqref{proof lemma monotone increasing 1} implies $\nu_n<\nu_{n+1}$, cf.~\eqref{definition of nu_n},  because $n^{-1}>(n+1)^{-1}$.
 
The inequality \eqref{proof lemma monotone increasing 1} is an immediate consequence of the Mean Value Theorem.
\end{proof}

Lemma~\ref{lemma monotone increasing} tells us that the sequence $\nu_n$
constructed in accordance with equations \eqref{definition of nu_0}--\eqref{lemma monotone increasing equation 1}
yields a partition of the positive semi-axis,
thus establishing  property (a).

Property (b) is established by the following Lemma.

\begin{lemma}
\label{lemma about achieving accuracy}
We have 
\begin{equation}
\label{lemma about achieving accuracy equation 1}
\nu_{n+1}-\nu_n =O(\nu_n^{-\frac{\alpha d}{s}}) \quad \text{as}\quad n\to+\infty.
\end{equation}
\end{lemma}
\begin{proof}
Formula \eqref{definition of nu_n} implies
\begin{equation}
\label{proof lemma about achieving accuracy equation 1}
\nu_{n+1}-\nu_n=O(n^{\beta-1})=O((\nu_n^{1/\beta})^{\beta-1})=O\bigl(\nu_n^{1-\frac1\beta}\bigr)\quad \text{as}\quad n\to+\infty.
\end{equation}
Combining \eqref{proof lemma about achieving accuracy equation 1} and \eqref{definition of beta} we obtain \eqref{lemma about achieving accuracy equation 1}.
\end{proof}

Suppose that $\lambda_k\in(\nu_n,\nu_{n+1}]$.
Using Theorem~\ref{Weyl asymptotics for elliptic systems theorem} we obtain
\begin{equation*}
\label{proof main theorem 6 equation 4bis extra}
\nu_n=b^{-s/d}k^{s/d}+o(k^{s/d})\quad\text{as}\quad k\to+\infty.
\end{equation*}
This gives us a different take on the statement of
Lemma~\ref{lemma about achieving accuracy} in that 
it allows us to equivalently recast 
\eqref{lemma about achieving accuracy equation 1}
as 
\begin{equation}
\label{proof main theorem 6 equation 5 extra}
\nu_{n+1}-\nu_n=O(k^{-\alpha}) \quad\text{as}\quad k\to+\infty.
\end{equation}

Finally, the following Lemma establishes that the $c_n$ can be chosen in such a way that our partition possesses property (c).

\begin{lemma}
\label{lemma about existence of gamma}
There exist constants $\gamma>1$ and $C>0$
such that, for a suitable choice of $c_n$ in \eqref{definition of nu_n} compatible with condition
\eqref{lemma monotone increasing equation 1}, we have
\begin{equation}
\label{lemma about existence of gamma equation 1} 
\operatorname{dist}\bigl(\nu_n\,,\textstyle \bigcup_{j=1}^{m^+}\sigma^+(A_j)\cup \sigma^+(A)\bigr)
\ge
Cn^{-\gamma}
\end{equation}
for all $n=1,2,\dots$.
\end{lemma}

\begin{proof}
In order to prove the lemma, it is enough to estimate from below the size of the largest gap in the set
\begin{equation}
\left(\textstyle \bigcup_{j=1}^{m^+}\sigma^+(A_j)\cup \sigma^+(A)\right)\cap\left[\nu_n-\tfrac\beta{4n},\nu_n+\tfrac\beta{4n}\right].
\end{equation}
To begin with, let us estimate from above the number of eigenvalues we can have in the interval $[\nu_n-\frac\beta{4n}, \nu_n+\frac\beta{4n}]$.

Theorem~\ref{Weyl asymptotics for elliptic systems theorem} tells us that
\begin{equation}
\label{proof lemma about existence of gamma equation 1}
N\left(\nu_n;\tfrac\beta{4n}\right)+\widetilde{N}\left(\nu_n;\tfrac\beta{4n}\right)
=o\bigl(n^{\frac{\beta d}{s}}\bigr).
\end{equation}

The quantity
\begin{equation*}
\label{proof lemma about existence of gamma equation 3}
\sup_{x\in \left[\nu_n-\tfrac\beta{4n}\,,\,\nu_n+\tfrac\beta{4n}\right]} \operatorname{dist}\left(x, \left(\textstyle \bigcup_{j=1}^{m^+}\sigma^+(A_j)\cup \sigma^+(A)\right)\cap\left[\nu_n-\tfrac\beta{4n},\nu_n+\tfrac\beta{4n}\right] \right)
\end{equation*}
is minimised when the eigenvalues $\lambda_k$ and $\mu_k$ are equidistributed in $\left[\nu_n-\tfrac\beta{4n},\nu_n+\tfrac\beta{4n}\right]$. As we are looking at an interval of length $\tfrac\beta{2n}\,$, formula
\eqref{proof lemma about existence of gamma equation 1} implies that we can choose $c_n$ such that
\eqref{lemma about existence of gamma equation 1} holds for
\begin{equation}
\label{admissible gammas}
\gamma= 1+\frac{d\beta}s\,.
\end{equation}
\end{proof}

The construction of our partition is summarised in Figure~\ref{figure partition}.

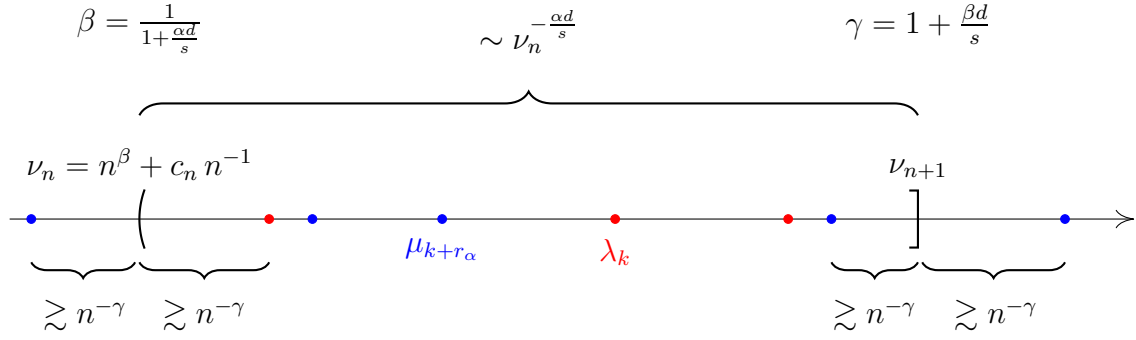
\begin{figure}
\centering
\begin{tikzpicture}[scale=1.15]

\draw[-{>[scale=3.5,
          length=2,
          width=2]}
] (0,0)--(13,0);

\draw [black,thick,domain=160:200] plot ({1*cos(\x)+2.5}, {1*sin(\x)});

\draw [thick] (10.4,-.33)--(10.5,-.33)--(10.5,.33)--(10.4,.33);

\node [above] at (1.5,.33) {$\nu_{n}=n^\beta+c_n\,n^{-1}$};
\node [above] at (10.5,.33) {$\nu_{n+1}$};

\draw[fill,red] (3,0) circle [radius=0.05];
\draw[fill,red] (7,0) circle [radius=0.05];
\draw[fill,red] (9,0) circle [radius=0.05];

\node [below,red] at (7,-.1) {$\lambda_k$};

\draw[fill,blue] (.25,0) circle [radius=0.05];
\draw[fill,blue] (3.5,0) circle [radius=0.05];
\draw[fill,blue] (5,0) circle [radius=0.05];
\draw[fill,blue] (9.5,0) circle [radius=0.05];
\draw[fill,blue] (12.2,0) circle [radius=0.05];

\node [below,blue] at (5,-.1) {$\mu_{k+r_\alpha}$} ;

\draw [thick, decorate,decoration={brace,amplitude=6pt,mirror}] (.25,-.5) -- (1.45,-.5);
\node [below] at (.87,-.8) {$\gtrsim n^{-\gamma}$};
\draw [thick, decorate,decoration={brace,amplitude=6pt,mirror}] (1.5,-.5) -- (3,-.5);
\node [below] at (2.25,-.8) {$\gtrsim n^{-\gamma}$};

\draw [thick, decorate,decoration={brace,amplitude=6pt,mirror}] (9.5,-.5) -- (10.5,-.5);
\node [below] at (10,-.8) {$\gtrsim n^{-\gamma}$};
\draw [thick, decorate,decoration={brace,amplitude=6pt,mirror}] (10.55,-.5) -- (12.2,-.5);
\node [below] at (11.4,-.8) {$\gtrsim n^{-\gamma}$};

\draw [thick, decorate,decoration={brace,amplitude=9pt}] (1.5,1.2) -- (10.5,1.2);
\node [above] at (6,1.8) {$\sim \nu_n^{-\frac{\alpha d}{s}}$};

\node [above] at (1.5,1.8) {$\beta=\frac{1}{1+\frac{\alpha d}{s}}$};
\node [above] at (10.5,1.9) {$\gamma=1+\frac{\beta d}{s}$};

\end{tikzpicture}
\caption{Construction of the partition of the positive semi-axis
}
\label{figure partition}
\end{figure}

\subsubsection{Preparatory lemmata}
\label{Preparatory lemmata}

We will now state and prove a few simple lemmata.

\begin{lemma}
\label{lemma about bound on number of ev in interval}
Let
$
\mu_1\le\mu_2\le\dots\le \mu_r
$
be real numbers and let $u_k$, $k=1,\ldots,r$, be an orthonormal set in $L^2(M)$. 
Suppose that
\begin{equation}
\label{lemma about bound on number of ev in interval equation 2}
\|(A-\mu_k)u_k\|_{L^2}\le \varepsilon, \qquad k=1,\ldots,r.
\end{equation}
Then 
\begin{equation}
\label{lemma about bound on number of ev in interval equation 3}
\# \{\lambda\ |\ \lambda\in \sigma(A)\cap [\mu_1-\sqrt{r}\,\varepsilon,\mu_r+\sqrt{r} \,\varepsilon] \} \ge  r.
\end{equation}
\end{lemma}

\begin{proof}
Without loss of generality, we can assume that $\mu_r \ge 0$ and $\mu_1=-\mu_r$. Arguing by contradiction, suppose \eqref{lemma about bound on number of ev in interval equation 3} is not true. 
Then one can choose a $u\in L^2(M)$ such that
\begin{equation}
\label{proof lemma about bound on number of ev in interval equation 1}
u=\sum_{k=1}^r a_k \,u_k,
\end{equation}
\begin{equation}
\label{proof lemma about bound on number of ev in interval equation 2}
\sum_{k=1}^r |a_k|^2=1,
\end{equation}
orthogonal to the eigenfunctions of $A$ corresponding to eigenvalues in $[-\mu_r-\sqrt{r}\,\varepsilon,\mu_r+\sqrt{r} \,\varepsilon]$.

On the one hand, the Spectral Theorem implies
\begin{equation}
\label{proof lemma about bound on number of ev in interval equation 3}
\|Au\|_{L^2}> \mu_r+\sqrt{r}\,\varepsilon.
\end{equation}

On the other hand, using formulae \eqref{lemma about bound on number of ev in interval equation 2}, \eqref{proof lemma about bound on number of ev in interval equation 1}, \eqref{proof lemma about bound on number of ev in interval equation 2}, the triangle inequality in $L^2(M)$ and the Cauchy--Schwarz inequality in $\mathbb{C}^r$, we obtain
\begin{equation}
\label{proof lemma about bound on number of ev in interval equation 4}
\begin{split}
\|Au \|_{L^2} 
&
\le 
\left\| \sum_{k=1}^r a_k (A-\mu_k)u_k \right\|_{L^2}+\left\|\sum_{k=1}^r a_k \mu_k u_k\right\|_{L^2}
\\
&
\le 
\sqrt{r}\, \varepsilon+ \mu_r.
\end{split}
\end{equation}
Formulae \eqref{proof lemma about bound on number of ev in interval equation 3} and \eqref{proof lemma about bound on number of ev in interval equation 4} give us a contradiction.
\end{proof}

\begin{lemma}
\label{lemma about almost orthogonality}
Let $u_k^{(j)}$ be a normalised eigenfunction of $A_j$ corresponding to the eigenvalue $\lambda_k^{(j)}$.
Then
\begin{equation}
\label{lemma about almost orthogonality equation 1}
\langle u_k^{(j)},u_{k'}^{(l)} \rangle= O(n^{-\infty})
\end{equation}
for $l\ne j$ and for all $k,k'$ such that $\lambda_{k}^{(j)}, \lambda_{k'}^{(l)}\in (\nu_n,\nu_{n+1}]$.
\end{lemma}

\begin{proof}
Using Proposition~\ref{proposition on action of projections on eigenfunctions}(b) and properties of pseudodifferential projections we get
\begin{equation*}
\label{proof lemma about almost orthogonality equation 1}
\langle u_k^{(j)},u_{k'}^{(l)} \rangle=\langle P_j u_k^{(j)},P_l u_{k'}^{(l)} \rangle+O(n^{-\infty})=\langle P_lP_j u_k^{(j)},u_{k'}^{(l)} \rangle+O(n^{-\infty})=O(n^{-\infty}).
\end{equation*}
\end{proof}

\begin{lemma}
\label{lemma numeric inequality}
Let $a_k, b_k$, $k=1,\ldots, r$, be nonnegative real numbers. Suppose that
\begin{equation}
\label{lemma numeric inequality equation 1}
\sum_{k=1}^r a_k \le C
\end{equation}
and
\begin{equation}
\label{lemma numeric inequality equation 2}
\sum_{k=1}^r b_n=1.
\end{equation}
Then there exists a $\widetilde{k}\in \{1,\ldots, r\}$ such that
\begin{equation}
\label{lemma numeric inequality equation 3}
a_{\widetilde{k}}\le C\, b_{\widetilde{k}}.
\end{equation}
\end{lemma}

\begin{proof}
Suppose 
\begin{equation}
\label{proof lemma numeric inequality equation 1}
a_{k}> C\, b_{k}\qquad \text{for all}\quad k=1,\ldots,r. 
\end{equation}
Then summing both sides of \eqref{proof lemma numeric inequality equation 1} over $k$ and using \eqref{lemma numeric inequality equation 2} we obtain
$
\sum_{k=1}^r a_k> C,
$
which contradicts \eqref{lemma numeric inequality equation 1}.
\end{proof}

In the proof of Theorem~\ref{Main results theorem 6} we will need to construct a set of orthonormal functions out of a set of $r$ functions which are only approximately orthonormal, up to an error that decays superpolynomially. The Gram--Schmidt process, as well as its standard modifications, yields a number of terms growing factorially with $r$, thus resulting in an overall error that is too big for our purposes.
The following lemma will give us an orthonormalisation procedure which circumvents  this shortcoming.

\begin{lemma}
\label{lemma about orthonormalisation}
Let $F$ be an Hermitian $r\times r$ matrix such that
\begin{equation}
\label{lemma about orthonormalisation equation 1}
\|F-I\|_{\max}\le \frac{1}{3r^2}\,,
\end{equation}
where $I$ is the $r\times r$ identity matrix and $\|F\|_{\max}:=\max_{1\le j,k \le r}|F_{jk}|$ is the max matrix norm. 
Then there exists an Hermitian matrix $G$ such that
\begin{equation}
\label{lemma about orthonormalisation equation 2}
GFG=I
\end{equation}
and
\begin{equation}
\label{lemma about orthonormalisation equation 3}
\|G-I\|_{\max}\le \|F-I\|_{\max}.
\end{equation}
\end{lemma}

\begin{proof}
It is easy to see that
under the condition \eqref{lemma about orthonormalisation equation 1}
the matrix $F$ is positive definite.
Put 
\begin{equation}
\label{proof lemma about orthonormalisation equation 1}
R:=F-I
\end{equation}
and define
\begin{equation}
\label{proof lemma about orthonormalisation equation 2}
G:=F^{-1/2}=\sum_{k=0}^{\infty} \begin{pmatrix}
-\frac12\\ k
\end{pmatrix} R^k.
\end{equation}
The series on the RHS of \eqref{proof lemma about orthonormalisation equation 2} converges
in the max matrix norm
as soon as $\|R\|_{\max}<r^{-1}$, which is guaranteed by \eqref{lemma about orthonormalisation equation 1}. Of course, \eqref{proof lemma about orthonormalisation equation 2} implies \eqref{lemma about orthonormalisation equation 2}.

Formulae \eqref{proof lemma about orthonormalisation equation 2}, \eqref{proof lemma about orthonormalisation equation 1} and \eqref{lemma about orthonormalisation equation 1} imply
\begin{equation*}
\label{proof lemma about orthonormalisation equation 3}
\begin{split}
\|G-I\|_{\max}
&
\le 
\sum_{k=1}^{\infty} \left|\begin{pmatrix}
-\frac12\\ k
\end{pmatrix}\right| \left\|R^k\right\|_{\max}
\le
\frac12 \left\|R\right\|_{\max}
+
\sum_{k=2}^{\infty} r^k \left\|R\right\|_{\max}^k
\\
&
\le
\frac12 \left\|R\right\|_{\max}
+
\frac{r^2 \left\|R\right\|_{\max}^2}{1-r\left\|R\right\|_{\max}}
\le
\frac12 \left\|R\right\|_{\max}
+
\frac{\frac1{3}\left\|R\right\|_{\max}}{1-\frac1{3r}}
\le
\left\|R\right\|_{\max}.
\end{split}
\end{equation*}
In the above calculation we used the weighted submultiplicative property of the max matrix norm, $\|R^k\|_{\max}\le r^k\|R\|_{\max}^k$.
\end{proof}

Lastly, we recall for the reader's convenience a fact from elementary functional analysis.

\begin{lemma}
\label{lemma about spectrum}
Let $B$ be a self-adjoint operator in a Hilbert space $(H, \|\cdot\|_H)$ with discrete spectrum.
Let $R$ be a positive number.
Let $\xi_k$, $k=1,\dots,r$, be all the eigenvalues of $B$ on the closed interval $[-R,R]$,
and let $v_k$ be the corresponding eigenfunctions.
Then for any $v\in D(B)$, $v\ne0$, satisfying
\begin{equation}
\label{lemma about spectrum equatoin 1}
\langle v_k,v\rangle_H=0,\qquad k=1,\dots,n,
\end{equation}
we have
\begin{equation}
\label{lemma about spectrum equatoin 2}
\|Bv\|_H>R\|v\|_H.
\end{equation}
\end{lemma}

\subsubsection{Proof of Theorem~\ref{Main results theorem 6}}
\label{Proof of Theorem 2.3}

We start by proving two propositions establishing
that, for sufficiently large $n$, we have the same number of eigenvalues $\lambda_k$ and $\mu_k$ in each interval $(\nu_n,\nu_{n+1}]$.

\begin{proposition}
\label{proposition easy direction}
There exists a natural number $K$ such that for all $n>K$ we have
\begin{equation}
\label{proposition easy direction equation 1}
N(\tfrac{\nu_n+\nu_{n+1}}{2};\tfrac{\nu_{n+1}-\nu_n}2)\ge \widetilde{N}(\tfrac{\nu_n+\nu_{n+1}}{2};\tfrac{\nu_{n+1}-\nu_n}2).
\end{equation}
\end{proposition}

\begin{proof}
Suppose $\tilde N(\tfrac{\nu_n+\nu_{n+1}}{2};\tfrac{\nu_{n+1}-\nu_n}2)=r$. This means that the interval $(\nu_n, \nu_{n+1}]$ contains precisely $r$ elements 
$
\mu_{p+1}\le \ldots\le \mu_{p+r}
$
from the sequence \eqref{positive eigenvalues of Aj combined} (recall that we have no eigenvalues in a neighbourhood of the endpoints of our partition). 

Each $\mu_{p+k}$, $k=1,\ldots, r$, is an eigenvalue of $A_j$ for some $j$. Let us denote by $v_k$, $k=1,\ldots, r$, the corresponding normalised eigenfunctions. Here we assume that eigenfunctions corresponding to the same operator $A_j$
are chosen to be orthogonal.

In view of Lemma~\ref{lemma about almost orthogonality} we have
\begin{equation}
\label{proof proposition easy direction equation 2}
\langle v_k, v_{k'} \rangle=\delta_{kk'}+O(n^{-\infty}),\qquad k,k'=1,\ldots, r,
\end{equation}
where $\delta_{kk'}$ is the Kronecker delta.

Let $F$ be the $r\times r$ matrix whose entries are defined in accordance with
\begin{equation}
\label{proof proposition easy direction equation 2a}
F_{kk'}:=\langle v_k, v_{k'} \rangle, \qquad 1\le k,k'\le r,
\end{equation}
and let $G$ be the matrix given by Lemma~\ref{lemma about orthonormalisation}.
Then formula \eqref{lemma about orthonormalisation equation 2} implies that the functions
\begin{equation}
\label{proof proposition easy direction equation 2b}
\tilde{v}_k:=\sum_{q=1}^r G_{qk} \,v_{q}, \qquad k=1,\ldots,r,
\end{equation}
satisfy
\begin{equation}
\label{proof proposition easy direction equation 3}
\langle \tilde v_k, \tilde v_{k'}\rangle=\delta_{kk'}, \qquad k,k'=1,\ldots,r.
\end{equation}
Furthermore, as $\|F-I\|_{\max}=O(n^{-\infty})$ in view of \eqref{proof proposition easy direction equation 2}--\eqref{proof proposition easy direction equation 2a}, formula \eqref{lemma about orthonormalisation equation 3}
and elliptic regularity imply
\begin{equation}
\label{proof proposition easy direction equation 4}
\tilde{v}_k=v_k+O(n^{-\infty}), \qquad k=1,\ldots, r.
\end{equation}

Formula \eqref{proof proposition easy direction equation 4} and Proposition~\ref{proposition on action of projections on eigenfunctions}(c) give us
$\|(A-\mu_{p+k})\tilde{v}_k \|_{L^2}=O(n^{-\infty})$, $k=1,\ldots,r$.
In particular, we have
\begin{equation}
\label{proof proposition easy direction equation 6}
\|(A-\mu_{p+k})\tilde{v}_k \|_{L^2}\le C n^{-\tilde{\gamma}}, \qquad k=1,\ldots,r,
\end{equation}
for some $C>0$ independent of $k$ and $n$, and $\tilde{\gamma}:=\tfrac{\beta d}{2s}+\gamma+1$.

In light of \eqref{proof proposition easy direction equation 3} and \eqref{proof proposition easy direction equation 6}, we can apply Lemma~\ref{lemma about bound on number of ev in interval} to obtain
\begin{equation}
\label{proof proposition easy direction equation 7}
\#\{\lambda_k\ |\ \lambda_k\in [\mu_{p+1}-C \sqrt{r}\,n^{-\tilde{\gamma}},\mu_{p+r}+C \sqrt{r}\,n^{-\tilde{\gamma}}]\}\ge r.
\end{equation}
As $\sqrt{r}=o\bigl(n^{\frac{\beta d}{2s}}\bigr)$, see
Theorem~\ref{Weyl asymptotics for elliptic systems theorem}, we have
$\sqrt{r}\,n^{-\tilde\gamma}=o(n^{-\gamma-1})$.
Hence Lemma~\ref{lemma about existence of gamma} tells us that
$
[\mu_{p+1}-C \sqrt{r}\,n^{-\tilde{\gamma}},\mu_{p+r}+C \sqrt{r}\,n^{-\tilde{\gamma}}] \subset (\nu_n, \nu_{n+1}]
$
for sufficiently large $n$, so that \eqref{proof proposition easy direction equation 7} implies \eqref{proposition easy direction equation 1}.
\end{proof}

\begin{proposition}
\label{proposition difficult direction}
There exists a natural number $K$ such that for all $n>K$ we have
\begin{equation}
\label{proposition difficult direction equation 1}
\widetilde{N}(\tfrac{\nu_n+\nu_{n+1}}{2};\tfrac{\nu_{n+1}-\nu_n}2)\ge N(\tfrac{\nu_n+\nu_{n+1}}{2};\tfrac{\nu_{n+1}-\nu_n}2).
\end{equation}
\end{proposition}

\begin{proof}
Let $u_k$ be orthonormal eigenfunctions of $A$ corresponding to $\lambda_k$ and, for each $j$, let $u_k^{(j)}$ be orthonormal eigenfunctions of $A_j$ corresponding to $\lambda_k^{(j)}$. 

Let us define
\begin{equation*}
\label{proof proposition difficult direction equation 1}
E_n:=\operatorname{span}(u_k\ |\ \lambda_k\in(\nu_n, \nu_{n+1}]),
\end{equation*}
\begin{equation*}
\label{proof proposition difficult direction equation 2}
\tilde E_n:=\operatorname{span}(u_k^{(j)}\ |\ \lambda_k^{(j)}\in(\nu_n, \nu_{n+1}], \ j=1, \ldots,m^+).
\end{equation*}

Arguing by contradiction, suppose
$\widetilde{N}(\tfrac{\nu_n+\nu_{n+1}}{2};\tfrac{\nu_{n+1}-\nu_n}2)< N(\tfrac{\nu_n+\nu_{n+1}}{2};\tfrac{\nu_{n+1}-\nu_n}2)$.
Then there exists a $u \in E_n$, $\|u\|_{L^2}=1$, such that
\begin{equation}
\label{proof proposition difficult direction equation 4}
\langle u,v \rangle=0 \qquad \forall v\in \tilde E_n.
\end{equation}

Proposition~\ref{proposition on action of projections on eigenfunctions}(a) implies
\begin{equation}
\label{proof proposition difficult direction equation 5}
u=\sum_{j=1}^{m^+}P_ju+O(n^{-\infty}),
\end{equation}
whereas Proposition~\ref{proposition on action of projections on eigenfunctions}(b) implies
\begin{equation}
\label{proof proposition difficult direction equation 6}
u_k^{(j)}=P_ju_k^{(j)}+O(n^{-\infty})
\end{equation}
for $j=1,\ldots,m^+$ and all $k$ such that $u_k^{(j)}\in \widetilde{E}_n$. 

We claim that
\begin{equation}
\label{proof proposition difficult direction equation 7}
\langle u_k^{(j')}, P_ju\rangle=O(n^{-\infty})
\end{equation}
for all $j,j'=1,\ldots,m^+$ and all $k$ such that $u_k^{(j)}\in \widetilde{E}_n$. Indeed, for $j\ne j'$ \eqref{proof proposition difficult direction equation 7} follows from \eqref{proof proposition difficult direction equation 6} and Theorem~\ref{theorem results from part 1}(a). For $j=j'$, using Theorem~\ref{theorem results from part 1}(a), \eqref{proof proposition difficult direction equation 6} and \eqref{proof proposition difficult direction equation 4}, we obtain
\begin{equation*}
\label{proof proposition difficult direction equation 8}
\langle u_k^{(j)}, P_ju\rangle
=\langle P_j u_k^{(j)}, u\rangle+O(n^{-\infty})
=\langle u_k^{(j)}, u\rangle+O(n^{-\infty})
=O(n^{-\infty}).
\end{equation*}

By the Spectral Theorem we have
\begin{equation}
\label{proof proposition difficult direction equation 9}
\left\| \left(A-\frac{\nu_{n+1}+\nu_{n}}2\right) u \right\|_{L^2} \le\frac{\nu_{n+1}-\nu_{n}}2.
\end{equation}
Squaring \eqref{proof proposition difficult direction equation 9}, substituting  \eqref{proof proposition difficult direction equation 5}  in and using Theorem~\ref{theorem results from part 1}(a), we get
\begin{equation}
\label{proof proposition difficult direction equation 10}
\sum_{j=1}^{m^+} \left\| \left(A-\frac{\nu_{n+1}+\nu_{n}}2\right) P_ju \right\|_{L^2}^2 \le\frac{(\nu_{n+1}-\nu_{n})^2}4+O(n^{-\infty})\,.
\end{equation}

Put
\begin{equation}
\label{proof proposition difficult direction equation 11bis}
\rho:=\beta+\gamma
\end{equation}
and let
$J:=\{j\in\{1,\ldots,m^+\}\ |\ \|P_ju\|_{L^2}\ge n^{-\rho} \}$.
As 
\begin{equation}
\label{proof proposition difficult direction equation 12}
\sum_{j=1}^{m^+}\|P_ju\|_{L^2}^2=1+O(n^{-\infty}),
\end{equation}
the set $J$ is nonempty. 

When we restrict the summation to indices from $J$, formula \eqref{proof proposition difficult direction equation 12} reads
\begin{equation*}
\label{proof proposition difficult direction equation 13}
\sum_{j\in J}\|P_ju\|_{L^2}^2=1+O(n^{-\rho}).
\end{equation*}
By suitably rescaling the function $u$, let us define a function $\tilde u$ such that
\begin{equation}
\label{proof proposition difficult direction equation 14}
\sum_{j\in J}\|P_j\tilde u\|_{L^2}^2=1.
\end{equation}
By restricting the summation to indices from the set $J$ only and expressing the result in terms of $\tilde u$, formula \eqref{proof proposition difficult direction equation 10} turns into
\begin{equation}
\label{proof proposition difficult direction equation 15}
\sum_{j\in J} \left\| \left(A-\frac{\nu_{n+1}+\nu_{n}}2\right) P_j\tilde u \right\|_{L^2}^2 \le \frac{(\nu_{n+1}-\nu_{n})^2}4(1+C\,n^{-\rho}).
\end{equation}
On account of \eqref{proof proposition difficult direction equation 14} and \eqref{proof proposition difficult direction equation 15}, Lemma~\ref{lemma numeric inequality} implies that there exists an $l\in J$ such that
\begin{equation}
\label{proof proposition difficult direction equation 15bis}
\left\| \left(A_l-\frac{\nu_{n+1}+\nu_{n}}2\right) P_l \tilde u \right\|_{L^2}^2 \le \frac{(\nu_{n+1}-\nu_{n})^2}4(1+C\,n^{-\rho}) \|P_l\tilde u\|_{L_2}^2.
\end{equation}
In \eqref{proof proposition difficult direction equation 15bis} we were able to replace $A$ with $A_l$ by resorting to formula \eqref{operators Aj}, which implies 
\begin{equation}
\label{AP_l = A_l}
AP_l=A_lP_l \mod \Psi^{-\infty}.
\end{equation}

Formula \eqref{proof proposition difficult direction equation 7} tells us that
\begin{equation}
\label{proof proposition difficult direction equation 16}
\langle u^{(l)}_k,P_l\tilde{u}\rangle=O(n^{-\infty})
\end{equation}
for all $k$ such that $u_k^{(l)}\in \tilde E_n$.
Put
\begin{equation}
\label{proof proposition difficult direction equation 17}
\widehat{u}_l:=
P_l\tilde{u}-\sum_{k\,:\,\lambda^{(l)}_k\in(\nu_n,\nu_{n+1}]}
\langle u^{(l)}_k,P_l\tilde{u}\rangle \,u^{(l)}_k,
\end{equation}
so that
\begin{equation}
\label{proof proposition difficult direction equation 18}
\langle u^{(l)}_k,\widehat{u}_l\rangle=0 \qquad \forall u_k^{(l)}\in \tilde{E}_n.
\end{equation}
Formulae
\eqref{proof proposition difficult direction equation 17}
and
\eqref{proof proposition difficult direction equation 16}
give us
$\widehat{u}_l=
P_l \tilde{u}+O(n^{-\infty})$,
so \eqref{proof proposition difficult direction equation 15bis} implies
\begin{equation}
\label{proof proposition difficult direction equation 20}
\left\| \left(A_l-\frac{\nu_{n+1}+\nu_{n}}2\right) \widehat{u}_l \right\|_{L^2} \le \frac{\nu_{n+1}-\nu_{n}}2(1+C\,n^{-\rho}) \|\widehat{u}_l\|_{L_2}.
\end{equation}

Formula \eqref{proof proposition difficult direction equation 11bis} and Lemma~\ref{lemma about achieving accuracy} imply
$(\nu_{n+1}-\nu_{n})\,n^{-\rho}=O(n^{-\gamma-1})$.
Therefore, in view of Lemma~\ref{lemma about existence of gamma}, for sufficiently large $n$ we have
\begin{equation}
\label{proof proposition difficult direction equation 21}
\widetilde{N}(\tfrac{\nu_n+\nu_{n+1}}{2};\tfrac{\nu_{n+1}-\nu_n}2)=\widetilde{N}(\tfrac{\nu_n+\nu_{n+1}}{2};\tfrac{(\nu_{n+1}-\nu_n)(1+C\,n^{-\rho})}2).
\end{equation}
But, on account of Lemma~\ref{lemma about spectrum}, formulae \eqref{proof proposition difficult direction equation 20} and \eqref{proof proposition difficult direction equation 18}
imply
\begin{equation*}
\label{proof proposition difficult direction equation 21bis}
\widetilde{N}(\tfrac{\nu_n+\nu_{n+1}}{2};\tfrac{(\nu_{n+1}-\nu_n)(1+C\,n^{-\rho})}2)
>
\widetilde{N}(\tfrac{\nu_n+\nu_{n+1}}{2};\tfrac{\nu_{n+1}-\nu_n}2),
\end{equation*}
which contradicts \eqref{proof proposition difficult direction equation 21}.
\end{proof}

Theorem~\ref{Main results theorem 6} now follows easily.

\begin{proof}[Proof of Theorem~\ref{Main results theorem 6}]
Propositions~\ref{proposition easy direction} and~\ref{proposition difficult direction} imply that there exists a natural $K$ such that
\begin{equation*}
\label{proof main theorem 6 equation 1}
N(\tfrac{\nu_n+\nu_{n+1}}{2};\tfrac{\nu_{n+1}-\nu_n}2)= \widetilde{N}(\tfrac{\nu_n+\nu_{n+1}}{2};\tfrac{\nu_{n+1}-\nu_n}2)
\end{equation*}
for all $n>K$. This means that for all $n>K$ each interval $(\nu_n, \nu_{n+1}]$ contains the same number of eigenvalues $\lambda_k$ from \eqref{positive eigenvalues of A} and eigenvalues $\mu_k$ from \eqref{positive eigenvalues of Aj combined}.

Let
\begin{equation*}
\label{proof main theorem 6 equation 2}
k':=\min\{k\ |\ \lambda_k\in (\nu_{K+1}, \nu_{K+2}]\},
\qquad 
k'':=\min\{k\ |\ \mu_k\in (\nu_{K+1}, \nu_{K+2}]\},
\end{equation*}
and put
$r_\alpha:=k''-k'$.
Then, using formula \eqref{proof main theorem 6 equation 5 extra}, we arrive at
\begin{equation*}
\label{proof main theorem 6 equation 5}
\lambda_{k}-\mu_{k+r_\alpha}=O(k^{-\alpha}) \quad\text{as}\quad k\to+\infty.
\end{equation*}
\end{proof}

\begin{remark}
Observe that in the proofs of Theorems~\ref{Main results theorem 4}--\ref{Main results theorem 6} choosing the operators $A_j$ precisely in accordance with formula \eqref{operators Aj} is not crucial. One can, for instance, replace \eqref{operators Aj} with
\begin{equation*}
A_j:=A-\underset{l\ne j}{\underset{l=1,\dots,m^+}\sum}c_{j,l}\,P_l^*AP_l\,,
\qquad j=1,\dots,m^+,
\end{equation*}
where $c_{j,l}>1$.
What we mainly relied upon is the fact that the operators $A_j$ satisfy the properties
\begin{equation*}
P_j^*A_jP_j=P_j^*AP_j \mod \Psi^{-\infty},
\end{equation*}
\begin{equation*}
[A_j, P_l]=0 \mod \Psi^{-\infty}\quad \text{for all}\quad l,
\end{equation*}
and
\begin{equation*}
P_l^*A_jP_l\le 0 \mod \Psi^{-\infty} \quad \text{for}\quad l \ne j.
\end{equation*}
\end{remark}

\section{Invariant subspaces in hyperbolic systems}
\label{Invariant subspaces in hyperbolic systems}

In this section we will apply our results to the study of first and second order hyperbolic systems.

\subsection{First order operators}
\label{First order operators}

Before addressing the proof of Theorem~\ref{Main results theorem 7}, let us recall, in an abridged manner and for the convenience of the reader, the propagator construction from \cite{dirac}, which builds upon \cite{LSV, SaVa, CDV} and is an extension to first order systems of earlier results for scalar operators \cite{wave,lorentzian}.

Let $A\in \Psi^1$ be an operator as in Section~\ref{Statement of the problem}.
For each $j\in\{-m^-,\ldots, -1,1,\ldots,m^+\}$ let us denote by $(x^{(j)}(t;y,\eta), \xi^{(j)}(t;y,\eta))$ the Hamiltonian flow in the cotangent bundle generated by the Hamiltonian $h^{(j)}(x,\xi)$, namely, the solution to Hamilton's equations
\begin{equation}
\label{hamilton's equations}
\begin{cases}
\dot{x}^{(j)}=h_{\xi}(x^{(j)},\xi^{(j)}), \\
\dot \xi^{(j)}=-h_{x}(x^{(j)},\xi^{(j)})
\end{cases}
\end{equation}
with initial condition $(x^{(j)}(0;y,\eta), \xi^{(j)}(0;y,\eta))=(y,\eta)$. Here and further on the dot denotes differentiation with respect to $t$ and subscripts denote partial differentiation.

For each $j$ choose a function $\varphi^{(j)}(t,x;y,\eta)\in C^\infty(\mathbb{R}\times M\times T'M; \mathbb{C})$ positively homogeneous in $\eta$ of degree 1 satisfying
\begin{enumerate}[(i)]
\item
$\left.\varphi^{(j)}\right|_{x=x^{(j)}}=0$,

\item
$\left. \varphi^{(j)}_{x^\alpha}\right|_{x=x^{(j)}}=\xi^{(j)}_\alpha$,

\item
$\left. \det\varphi^{(j)}_{x^\alpha\eta_\beta}\right|_{x=x^{(j)}}\ne0$,

\item
$\operatorname{Im} \varphi^{(j)}\ge 0$.
\end{enumerate}
Such functions are called \emph{phase functions} and
they always exist \cite[Lemma~1.4]{LSV}.

Then the propagator 
$
U(t):=e^{-i t A}
$
 can be written, modulo an infinitely smoothing operator, as the sum of precisely $m$ oscillatory integrals
\begin{equation}
\label{U(t) as sum of U(j)j}
U(t)= \sum_j U^{(j)}(t) \mod C^\infty(\mathbb{R};\Psi^{-\infty})
\end{equation}
where
\begin{equation}
\label{oscillatory integral U j}
[U^{(j)}(t)u](x)=\frac{1}{(2\pi)^d}\int_{T'M} e^{i \varphi^{(j)}(t,x;y,\eta)}\,\mathfrak{a}^{(j)}(t;y,\eta)\,\chi^{(j)}(t,x;y,\eta)\,w^{(j)}(t,x;y,\eta)\,u(y)\,dy\,d\eta
\end{equation}
and
\begin{itemize}

\item the function $\chi^{(j)}\in C^\infty(\mathbb{R}\times M \times T'M)$ is a cut-off satisfying
\begin{enumerate}[(a)]
\item $\chi^{(j)}(t,x;y,\eta)=0$ on $\{(t,x;y,\eta) \,|\, |h^{(j)}(y,\eta)|\leq 1/2\}$,
\item $\chi^{(j)}(t,x;y,\eta)=1$ on the intersection of $\{(t,x;y,\eta) \,|\, |h^{(j)}(y,\eta)| \geq 1\}$ with some conical neighbourhood of $\{(t,x^{(j)}(t;y,\eta);y,\eta) \}$,
\item $\chi^{(j)}(t,x;y,\alpha\, \eta)=\chi^{(j)}(t,x;y,\eta)$ for $\alpha\geq 1$ on $\{ (t,x;y,\eta) \, | \, |h^{(j)}(y,\eta)|\geq 1   \}$;
\end{enumerate}

\item
the weight $w^{(j)}$ is defined by the phase function $\varphi^{(j)}$ in accordance with
\begin{equation*}
\label{weight definition general case}
w^{(j)}(t,x;y,\eta):= \left[{\det}^2( \varphi^{(j)}_{x^\alpha\eta_\beta}) \right]^\frac14,
\end{equation*}
with the smooth branch of the complex root chosen in such a way that $w^{(j)}(0,y;y,\eta)=1$.
\end{itemize}

The smooth matrix-function $\mathfrak{a}^{(j)}\in \mathrm{S}_{\mathrm{ph}}^0(\mathbb{R}\times T'M;\mathrm{Mat}(m;\mathbb{C}))$ appearing in \eqref{oscillatory integral U j} is the unknown in the algorithm for the construction of $U^{(j)}(t)$. It is an element in the class of poly\-homogeneous symbols of order zero with values in $m\times m$ complex matrices, which means that $\mathfrak{a}^{(j)}$ admits an asymptotic expansion in components positively homogeneous in momentum,
\begin{equation*}
\label{asympotic expansion symbol for A}
\mathfrak{a}^{(j)}(t;y,\eta)\sim \sum_{k=0}^{+\infty} \mathfrak{a}^{(j)}_{-k}(t;y,\eta), \qquad \mathfrak{a}^{(j)}_{-k}(t;y,\alpha\,\eta)=\alpha^{-k}\, \mathfrak{a}^{(j)}_{-k}(t;y,\eta), \quad\forall \alpha >0.
\end{equation*}
The symbol $\mathfrak{a}^{(j)}$ is determined by the requirement that $U^{(j)}(t)$ satisfies, in a distributional sense, the hyperbolic equation
\begin{equation}
\label{equation satisfied by U j}
(-i\partial_t+A)U^{(j)}(t)=0\mod C^\infty(\mathbb{R}; \Psi^{-\infty}).
\end{equation}
Note that $\mathfrak{a}^{(j)}$ does not depend on $x$: this is achieved by means of a procedure called
\emph{reduction of the amplitude}, which turns the partial differential equations brought about by \eqref{equation satisfied by U j} into a hierarchy of transport equations for the homogeneous components of $\mathfrak{a}^{(j)}$ --- ordinary differential equations
in the variable $t$ --- which can be solved iteratively. We refer the reader to \cite[Section~3]{dirac} for further details.

The initial conditions for the transport equations are obtained from the initial condition for the propagator itself:
\begin{equation}
\label{initial condition}
\sum_j U^{(j)}(0)=\operatorname{Id} \mod \Psi^{-\infty}.
\end{equation}

Clearly, the oscillatory integrals $U^{(j)}(0)$ define pseudodifferential operators in $\Psi^{0}$. Furthermore formula \eqref{initial condition} tells us that the oscillatory integrals $U^{(j)}(t)$ for different $j$'s are not independent, but they are related to one another via the initial condition.

By examining formulae \eqref{oscillatory integral U j}, \eqref{equation satisfied by U j} and \eqref{initial condition} it is not difficult to see that $(U^{(j)}(0))_\mathrm{prin}=P^{(j)}$.
This turns out not to be a coincidence: the following theorem will allow us to establish a relation between the pseudodifferential operators $U^{(j)}(0)$ and our pseudodifferential projections~$P_j\,$.

\begin{theorem}
\label{theorem initial condition vs projection}
Put
\begin{equation}
\label{theorem initial condition vs projection equation 1}
V_{lj}(t):=P_l U^{(j)}(t).
\end{equation}
Then for all $j,l\in\{-m^-, \ldots, -1, 1, \ldots, m^+\}$, $j\ne l$, we have
\begin{equation}
\label{theorem initial condition vs projection equation 2}
V_{lj}(t)=0 \mod C^\infty(\mathbb{R};\Psi^{-\infty}).
\end{equation}
\end{theorem}

\begin{proof}
By the definition of $U^{(j)}(t)$ and the fact that $[A,P_l]\in \Psi^{-\infty}$, the operator \eqref{theorem initial condition vs projection equation 1} satisfies 
\begin{equation}
\label{proof theorem initial condition vs projection equation 1}
(-i\partial_t+A)V_{lj}(t)=0  \mod C^\infty(\mathbb{R};\Psi^{-\infty}).
\end{equation}
Arguing by contradiction, suppose that \eqref{theorem initial condition vs projection equation 2} is false. Then there exists an integer $k\ge0$ such that 
\begin{equation}
\label{proof theorem initial condition vs projection equation 2}
V_{lj}(t)\in C^\infty(\mathbb{R};\Psi^{-k}), \qquad V_{lj}(t)\not\in C^\infty(\mathbb{R};\Psi^{-k-1}).
\end{equation}
Let $(V_{lj})_\mathrm{prin,k}(t;y,\eta)$ be the principal symbol of $V_{lj}(t)$ as an operator in $C^\infty(\mathbb{R};\Psi^{-k})$, cf.~\cite[Definition~3.7]{dirac}.

A simple analysis of the leading transport equation for the homogeneous components of the symbol of $V_{lj}(t)$ arising from \eqref{proof theorem initial condition vs projection equation 1} --- see, e.g., \cite[subsection~3.3.4]{nicoll} --- tells us that \eqref{proof theorem initial condition vs projection equation 1} can be satisfied only if
\begin{equation}
\label{proof theorem initial condition vs projection equation 3}
[P^{(j)}(x^{(j)}(t;y,\eta),\xi^{(j)}(t;y,\eta))][(V_{lj})_\mathrm{prin,k}(t;y,\eta)]=[(V_{lj})_\mathrm{prin,k}(t;y,\eta)].
\end{equation}

Now, formula \eqref{theorem initial condition vs projection equation 1}
and the idempotency property of pseudodifferential projections imply
\begin{equation}
\label{proof theorem initial condition vs projection equation 4}
P_l V_{lj}(t)=V_{lj}(t) \mod C^\infty(\mathbb{R};\Psi^{-\infty}).
\end{equation}
Computing the principal symbol of the LHS of \eqref{proof theorem initial condition vs projection equation 4}, we conclude that \eqref{proof theorem initial condition vs projection equation 4} can only be satisfied if
\begin{equation}
\label{proof theorem initial condition vs projection equation 5}
[P^{(l)}(x^{(j)}(t;y,\eta),\xi^{(j)}(t;y,\eta))][(V_{lj})_\mathrm{prin,k}(t;y,\eta)]=[(V_{lj})_\mathrm{prin,k}(t;y,\eta)].
\end{equation}
In writing \eqref{proof theorem initial condition vs projection equation 5} we used the standard formula for the action of a pseudodifferential operator on an exponent \cite[\S18]{shubin} and the fact that $V_{lj}(t)$ is an oscillatory integral with phase function $\varphi^{(j)}$.

As $P^{(j)}P^{(l)}=0$ pointwise in $T^*M \setminus\{0\}$, formulae \eqref{proof theorem initial condition vs projection equation 3} and \eqref{proof theorem initial condition vs projection equation 5} imply
\begin{equation*}
\label{proof theorem initial condition vs projection equation 6}
(V_{lj})_\mathrm{prin,k}(t;y,\eta)=0 \quad \text{for all}\quad t\in \mathbb{R},\ (y,\eta)\in T'M\setminus\{0\},
\end{equation*}
which contradicts \eqref{proof theorem initial condition vs projection equation 2}.
\end{proof}

\begin{corollary}
\label{corollary initial condition vs projection}
For every $j\in\{-m^-, \ldots, -1, 1, \ldots, m^+\}$ we have
\begin{equation}
\label{corollary initial condition vs projection equation 1}
U^{(j)}(0)=P_j \mod \Psi^{-\infty}.
\end{equation}
\end{corollary}

\begin{proof}
Theorem~\ref{theorem results from part 1}(a)(iv) and formula \eqref{initial condition} imply
\begin{equation}
\label{proof corollary initial condition vs projection equation 1}
\sum_j U^{(j)}(0)=\sum_j P_j \mod \Psi^{-\infty},
\end{equation}
whereas
Theorem~\ref{theorem results from part 1}(a)(iv)
and
Theorem~\ref{theorem initial condition vs projection} imply
\begin{equation}
\label{proof corollary initial condition vs projection equation 2}
U^{(j)}(0)=P_j U^{(j)}(0) \mod \Psi^{-\infty}.
\end{equation}
Acting with $P_l$ on the left in \eqref{proof corollary initial condition vs projection equation 1} and using, once again,
Theorem~\ref{theorem initial condition vs projection},
we obtain
\begin{equation}
\label{proof corollary initial condition vs projection equation 3}
P_lU^{(l)}(0)=P_l \mod \Psi^{-\infty}.
\end{equation}
Combining \eqref{proof corollary initial condition vs projection equation 2} and \eqref{proof corollary initial condition vs projection equation 3} we arrive at \eqref{corollary initial condition vs projection equation 1}.
\end{proof}

We are now in a position to prove Theorem~\ref{Main results theorem 7}.

\begin{proof}[Proof of Theorem~\ref{Main results theorem 7}]
The first equality, namely
$U^{(j)}(t)=P_jU(t) \mod C^{\infty}(\mathbb{R};\Psi^{-\infty})$,
follows immediately from \eqref{U(t) as sum of U(j)j} and Theorem~\ref{theorem initial condition vs projection}.

Let us prove the second equality, namely
\begin{equation}
\label{Proof of Theorem 2.4 equation 1}
P_j U(t)=U(t)P_j \mod C^{\infty}(\mathbb{R};\Psi^{-\infty}).
\end{equation}

Put
$R(t):=U(-t)P_j U(t)-P_j$.
Then, in view of \eqref{definition of propagator}, $R(t)$ satisfies 
\begin{equation}
\label{Proof of Theorem 2.4 equation 3}
\partial_t R(t)=-i\left(  -AU(-t)P_jU(t)+U(-t)P_jAU(t) \right),
\end{equation}
\begin{equation}
\label{Proof of Theorem 2.4 equation 4}
R(0)=0.
\end{equation}
As $A$ commutes with $U(-t)$ and, modulo $\Psi^{-\infty}$, with $P_j$, formula \eqref{Proof of Theorem 2.4 equation 3} implies
\begin{equation}
\label{Proof of Theorem 2.4 equation 5}
\partial_t R(t)=iU(-t)[A,P_j]U(t)=0\mod C^\infty(\mathbb{R};\Psi^{-\infty}).
\end{equation}
Combining \eqref{Proof of Theorem 2.4 equation 5} and \eqref{Proof of Theorem 2.4 equation 4} we obtain
$R(t)=0 \mod C^{\infty}(\mathbb{R};\Psi^{-\infty})$,
which gives us \eqref{Proof of Theorem 2.4 equation 1}.
\end{proof}

In fact, Theorem~\ref{theorem initial condition vs projection} and Theorem~\ref{Main results theorem 7} imply the following stronger result.

\begin{corollary}
\label{corollary commuation of U(j) and P_j}
We have
\begin{equation}
\label{corollary commuation of U(j) and P_j equation 0}
P_j\,U^{(j)}(t)=U^{(j)}(t)P_j=U^{(j)}(t) \mod C^{\infty}(\mathbb{R};\Psi^{-\infty}) \quad \text{for all}\quad j
\end{equation}
and
\begin{equation}
\label{corollary commuation of U(j) and P_j equation 1}
P_l\,U^{(j)}(t)=U^{(j)}(t)P_l=0 \mod C^{\infty}(\mathbb{R};\Psi^{-\infty}) \quad \text{for}\quad l\ne j.
\end{equation}
\end{corollary}

\begin{proof}
Theorem~\ref{theorem initial condition vs projection} implies
\begin{equation}
\label{proof corollary commuation of U(j) and P_j equation 0}
U^{(j)}(t)=P_j\,U^{(j)}(t) \mod C^{\infty}(\mathbb{R};\Psi^{-\infty}).
\end{equation}
Substituting \eqref{U(t) as sum of U(j)j} into \eqref{Proof of Theorem 2.4 equation 1} and using \eqref{proof corollary commuation of U(j) and P_j equation 0} we obtain
\begin{equation}
\label{proof corollary commuation of U(j) and P_j equation 1}
U^{(j)}(t)=\sum_l P_l \,U^{(l)}(t) P_j \mod C^{\infty}(\mathbb{R};\Psi^{-\infty}).
\end{equation}
Multiplying \eqref{proof corollary commuation of U(j) and P_j equation 1} by $P_j$ on the left and using \eqref{proof corollary commuation of U(j) and P_j equation 0} once again we get
\begin{equation}
\label{proof corollary commuation of U(j) and P_j equation 2}
P_j U^{(j)}(t)=U^{(j)}(t)P_j \mod C^{\infty}(\mathbb{R};\Psi^{-\infty}).
\end{equation}
Formulae \eqref{proof corollary commuation of U(j) and P_j equation 0} and \eqref{proof corollary commuation of U(j) and P_j equation 2} give us \eqref{corollary commuation of U(j) and P_j equation 0}.

Now, formula \eqref{corollary commuation of U(j) and P_j equation 0}
implies that for $l\ne j$ we have
\begin{equation}
\label{proof corollary commuation of U(j) and P_j equation 3}
U^{(j)}(t)P_l=U^{(j)}(t)P_jP_l=0\mod C^{\infty}(\mathbb{R};\Psi^{-\infty}).
\end{equation}
Formula \eqref{proof corollary commuation of U(j) and P_j equation 3} and Theorem~\ref{theorem initial condition vs projection} give us \eqref{corollary commuation of U(j) and P_j equation 1}.
\end{proof}

\begin{remark}
\label{remark about weaker result from Dirac}
Note that a weaker version of Theorem~\ref{Main results theorem 7} was, effectively, obtained in \cite[Section~3]{dirac}. More precisely, it was shown that
\begin{equation*}
\label{remark about weaker result from Dirac equation 1}
\sum_{j=1}^{m^+} U^{(j)}(t)=\sum_{j=1}^{m^+} P_j U(t) \mod C^\infty(\mathbb{R};\Psi^{-\infty}),
\end{equation*}
\begin{equation*}
\label{remark about weaker result from Dirac equation 2}
\sum_{j=-m^-}^{-1} U^{(j)}(t)=\sum_{j=-m^-}^{-1} P_j U(t) \mod C^\infty(\mathbb{R};\Psi^{-\infty}).
\end{equation*}
To see that this is the case, one needs to combine \cite[Theorem~3.3]{dirac}  with \cite[Theorem~2.7]{part1}.

Note also that in \cite[Sec.~5]{bolte} the authors analysed, in a similar spirit,
the localisation of the propagator in a given spectral window of the operator $A$, albeit in a somewhat different setting.  
The use of pseudodifferential projections in the study of the unitary evolution for matrix operators was employed in \cite{brummelhuis} as well, in the semiclassical setting and under additional assumptions on $A$, in the context of Egorov-type theorems. See also \cite{cordes3}.
\end{remark}

Theorem~\ref{Main results theorem 7} tells us that pseudodifferential projections decompose $L^2(M)$ into almost-orthogonal almost-invariant subspaces under the unitary time evolution. Namely, if $v\in P_j L^2(M)$ then
\begin{equation*}
\label{remark about weaker result from Dirac equation 3}
U(t)v=U^{(j)}(t)v  \mod C^\infty(\mathbb{R}\times M),
\end{equation*}
\begin{equation*}
\label{remark about weaker result from Dirac equation 4}
U^{(l)}(t)v =0  \mod C^\infty(\mathbb{R}\times M) \quad \text{for}\quad l\ne j.
\end{equation*}

\subsection{Nonnegative second order operators}
\label{Nonnegative second order operators}

In this subsection we will show that one can obtain results analogous to those from subsection~\ref{First order operators} for nonnegative second order operators.

\

Let $A\in \Psi^{2}$ be a nonnegative self-adjoint elliptic operator and suppose that its principal symbol has simple eigenvalues. As explained in Section~\ref{Main results}, we define its propagator to be
\begin{equation}
\label{propagator second order}
U(t):=e^{-it \sqrt{A}}.
\end{equation}
The fact that $\sqrt{A}$ is a well-defined pseudodifferential operator follows, for example, from \cite{seeley}. 
The unitary operator \eqref{propagator second order} is the solution operator of the first order hyperbolic pseudo\-differential system
$(-i\partial_t+\sqrt{A})f=0$, 
subject to the initial condition $f|_{t=0}=f_0$.
Of course, the knowledge of $U(t)$ is sufficient for the construction of the general solution of the second order hyperbolic system
$(\partial^2_t -A)f=0$,
subject to initial conditions $f|_{t=0}=f_0$, $(\partial_t f)|_{t=0}=f_1$. Indeed, we have
\begin{equation*}
\label{general solution second order equation}
f=\cos(t\,\sqrt{A})f_0+ A^{-1/2}\sin(t\,\sqrt{A})f_1+t \sum_{k\,:\,\lambda_k=0}\langle u_k, f_1\rangle,
\end{equation*}
where $A^{-1/2}$ is the pseudoinverse of $\sqrt{A}$ \cite[Ch.~2 Sec.~2]{rellich}, $2 \cos(t\,\sqrt{A})=U(t)+U(t)^*$ and $2i \sin(t\,\sqrt{A})=U(t)^*-U(t)$.

Let $h^{(j)}(x,\xi)$, $j=1,\ldots, m$, be the eigenvalues of $(\sqrt{A})_\mathrm{prin}=\sqrt{A_\mathrm{prin}}$. Clearly, the $h^{(j)}$'s are positively homogeneous in momentum $\xi$ of degree 1, strictly positive and distinct. This follows from the fact that $A$ is nonnegative, elliptic and $A_\mathrm{prin}$ has simple eigenvalues. Let $\varphi^{(j)}$, $j=1,\ldots, m^+$ be phase functions satisfying conditions $(i)$--$(iv)$ from subsection \ref{First order operators}. 

Then the operator $U(t)$ can be constructed explicitly, modulo $C^\infty(\mathbb{R};\Psi^{-\infty})$, as the sum of $m$ oscillatory integrals $U^{(j)}(t)$ of the form \eqref{oscillatory integral U j}. Remarkably, the amplitude of the $U^{(j)}(t)$ can be determined without the need of extracting the square root of $A$. Indeed, one can retrace the construction algorithm outlined in subsection \ref{First order operators} replacing \eqref{equation satisfied by U j} with
\begin{equation}
\label{equation satisfied by U j second order}
(\partial_t^2-A)U^{(j)}(t)=0 \mod C^{\infty}(\mathbb{R};\Psi^{-\infty}).
\end{equation}
The use in \eqref{equation satisfied by U j second order} of the second order operator $\partial_t^2-A$  as opposed to its ``half-wave'' version $-i\,\partial_t+\sqrt{A}$  is justified by \cite[Theorem 3.2.1]{SaVa}.

\begin{proposition}
\label{proposition projections A vs root A}
Let $P_j$ and $\tilde P_j$, $j=1, \ldots,m$, be the pseudodifferential projections uniquely determined, modulo $\Psi^{-\infty}$, by $A$ and $\sqrt{A}$ respectively, in accordance with Theorem~\ref{theorem results from part 1}(a). Then
\begin{equation}
\label{projections A vs root A equation 1}
P_j=\tilde P_j \mod \Psi^{-\infty}
\end{equation}
for all $j$.
\end{proposition}

\begin{proof}
We have
\begin{equation}
\label{proof projections A vs root A equation 1}
A_\mathrm{prin}=\sum_{l=1}^m h^{(l)} P^{(l)}
\end{equation}
and
\begin{equation*}
\label{proof projections A vs root A equation 2}
(\sqrt{A})_\mathrm{prin}=\sum_{l=1}^m \sqrt{h^{(l)}} P^{(l)},
\end{equation*}
therefore
$(P_j)_\mathrm{prin}=(\tilde P_j)_\mathrm{prin}=P^{(j)}$.
As $[\tilde P_j, \sqrt{A}]=0 \mod\Psi^{-\infty}$, clearly $[\tilde P_j,A]=0\mod \Psi^{-\infty}$. Then the identity \eqref{projections A vs root A equation 1} follows from \cite[Theorem~4.1]{part1} and the uniqueness of pseudodifferential projections.
\end{proof}

It is worth remarking that the claim of Proposition~\ref{proposition projections A vs root A} is a nontrivial property of our pseudo\-differential projections which cannot be obtained by simply looking at the functional calculus of $A$ and $\sqrt{A}$.

\begin{proof}[Proof of Theorem~\ref{Main results theorem 8}]
Equation \eqref{Main results theorem 8 equation 1} follows immediately by applying Theorem~\ref{Main results theorem 7} to $\sqrt{A}$ and using Proposition~\ref{proposition projections A vs root A}.
\end{proof}

\section{Refined spectral asymptotics}
\label{Refined spectral asymptotics}

Theorems~\ref{Main results theorem 4}--\ref{Main results theorem 6} open the way to computing spectral asymptotics for each of the $m$ families which the spectrum of $A$ partitions into. We will provide here a brief description of how the results of this paper can be used to refine our understanding of results available in the literature, focussing on first order operators. Further on we assume that $A\in \Psi^{1}$.

\begin{remark}
One could, in principle, perform the forthcoming argument for nonnegative operators of even order, but this would require a lengthy discussion and would substantially increase the size of the paper. For this reason we decided to refrain from discussing refined spectral asymptotics in greater generality in the current paper.
\end{remark}

Let $N^+(\lambda)$ be the positive counting function of $A$ and let $N^+_j(\lambda)$ be the positive counting function of $A_j$, $j=1,\ldots, m^+$, defined in accordance with \eqref{Weyl asymptotics for elliptic systems counting function definition}.
Establishing a precise relation between $N^+(\lambda)$ and the $N_j^+(\lambda)$, $j=1,\ldots, m^+$, is a challenging task, and it is not \emph{a priori} clear whether a simple quantitative relation can be established in the general case. The issue at hand is that we are dealing with discontinuous functions which can experience massive jumps in the presence of spectral clusters.

What one can do is to establish a relation between the Weyl coefficients of $N^+(\lambda)$ and those of the $N_j^+(\lambda)$. 

Let
\begin{equation*}
\label{local conting function A^+}
N^+(x;\lambda):=
\begin{cases}
0 & \text{for}\ \lambda\le 0,\\
\sum_{k\,:\,0<\lambda_k<\lambda} [v_k(x)]^* \,v_k(x) &\text{for}\ \lambda> 0,\\
\end{cases}
\end{equation*}
be the positive \emph{local} counting function of $A$.
In an analogous manner, we define positive \emph{local} counting functions $N_j^+(x;\lambda)$ for each of the $A_j$, $j=1,\ldots,m^+$.

Let $\widehat{\mu}:\mathbb{R}\to \mathbb{C}$ be a smooth function such that $\widehat{\mu}= 1$ in some neighbourhood of the origin and 
$\operatorname{supp} \widehat{\mu}\subset(-T_0,T_0)$,
where $T_0$ is the infimum of lengths of all the Hamiltonian loops (see~\eqref{hamilton's equations})
originating from all the points of the manifold. Let $\mu$ be the inverse Fourier transform of $\widehat{\mu}$, where we adopt the convention
\begin{equation*}
\mathcal{F}[\mu](t)=\widehat{\mu}(t)=\int_{-\infty}^{+\infty} e^{-it\lambda} \mu(\lambda) \,d\lambda,
\qquad
\mathcal{F}^{-1}[\widehat{\mu}](\lambda)=\mu(\lambda)=\dfrac{1}{2\pi}\int_{-\infty}^{+\infty} e^{it\lambda}\,\widehat{\mu\,}(t) \,dt,
\end{equation*}
for the Fourier transform and inverse Fourier transform, respectively.

 It is known \cite{DuGu,Ivr80,Ivr84,Ivr98,SaVa} that the mollified derivative of the positive local counting function admits a complete asymptotic expansion in integer powers of $\lambda$:
\begin{equation}
\label{expansion for mollified derivative of counting function}
((N^+)'*\mu)(x,\lambda)=
a_{d-1}(x)\,\lambda^{d-1}
+
a_{d-2}(x)\,\lambda^{d-2}
+\dots
\quad
\text{as}
\quad
\lambda\to+\infty,
\end{equation}
\begin{equation}
\label{expansion for mollified derivative of counting function with j}
((N^+_j)'*\mu)(x,\lambda)=
a_{d-1}^{(j)}(x)\,\lambda^{d-1}
+
a_{d-2}^{(j)}(x)\,\lambda^{d-2}
+\dots
\quad
\text{as}
\quad
\lambda\to+\infty.
\end{equation}
Here $*$ stands for the convolution in the variable $\lambda$ and the prime stands for differentiation with respect to $\lambda$. The functions appearing as coefficients of powers of $\lambda$ in the asymptotic expansions \eqref{expansion for mollified derivative of counting function} and \eqref{expansion for mollified derivative of counting function with j} are called \emph{Weyl coefficients}.

\begin{proposition}
\label{proposition U(j) vs U_Aj}
For $j\in\{1,\ldots, m^+\}$ we have
\begin{equation}
\label{proposition U(j) vs U_Aj equation 1}
U_{A_j}^+(t)=U^{(j)}(t) \mod C^{\infty}(\mathbb{R};\Psi^{-\infty}),
\end{equation}
where $U_{A_j}^+(t):=\theta(A_j)\,e^{-it A_j}$ is the positive propagator of the operator $A_j$.
\end{proposition}

\begin{proof}
Let $P_l$ and $\widetilde P_l$ be the pseudodifferential projections associated with $A$ and $A_j$ respectively, in accordance with Theorem~\ref{theorem results from part 1}(a). Recalling \eqref{operators Aj} and arguing as in the proof of Proposition~\ref{proposition projections A vs root A}, it is easy to see that the $\tilde{P}_l$ are just a reshuffling of the $P_l$. In particular, we have
\begin{equation}
\label{proof proposition U(j) vs U_Aj equation 1}
P_j=\tilde P_1 \mod \Psi^{-\infty}.
\end{equation}
Indeed, $(A_j)_\mathrm{prin}$ has only one positive eigenvalue, $h^{(j)}$, see~\eqref{Aj prin}.

Then, formula \eqref{proof proposition U(j) vs U_Aj equation 1} and Corollary~\ref{corollary initial condition vs projection} imply
\begin{equation}
\label{proof proposition U(j) vs U_Aj equation 2}
U_{A_j}^+(0)=U^{(j)}(0)=P_j \mod  \Psi^{-\infty}.
\end{equation}

Substituting \eqref{corollary commuation of U(j) and P_j equation 0} into \eqref{equation satisfied by U j} we obtain
\begin{equation*}
(-i\partial_t+AP_j)U^{(j)}(t)=0\mod C^\infty(\mathbb{R}; \Psi^{-\infty}).
\end{equation*}
In view of formula \eqref{AP_l = A_l}, the above equation can be recast as
\begin{equation}
\label{proof proposition U(j) vs U_Aj equation 4}
(-i\partial_t+A_j)U^{(j)}(t)=0\mod C^\infty(\mathbb{R}; \Psi^{-\infty}).
\end{equation}

Now, $U^+_{A_j}(t)$ also satisfies \eqref{proof proposition U(j) vs U_Aj equation 4} by definition.
Thus, $U^{(j)}(t)$ and $U^+_{A_j}(t)$ satisfy the same first order hyperbolic  equation \eqref{proof proposition U(j) vs U_Aj equation 4} with the same initial condition \eqref{proof proposition U(j) vs U_Aj equation 2}. This gives us~\eqref{proposition U(j) vs U_Aj equation 1}.
\end{proof}

\begin{theorem}
\label{theorem refined spectral asymptotics}

\begin{enumerate}[(a)]
\item
We have
\begin{equation}
\label{theorem refined spectral asymptotics equation 1}
((N^+_j)'*\mu)(x,\lambda)=\mathcal{F}^{-1}[\operatorname{tr} u^{(j)}(t,x,x) \,\widehat{\mu}(t)]+O(\lambda^{-\infty})\quad \text{as}\quad \lambda\to+\infty,
\end{equation}
where $u(t,x,y)$ is the Schwartz kernel of $U^{(j)}(t)$ and $\operatorname{tr}$ stands for the matrix trace.

\item 
The Weyl coefficients of $A$ and $A_j$, $j=1,\ldots, m^+$, are related as
\begin{equation}
\label{theorem refined spectral asymptotics equation 2}
a_k(x)=\sum_{j=1}^{m^+} a_k^{(j)}(x), \qquad k=d-1,d-2,\ldots.
\end{equation}

\item
The first two Weyl coefficients of $A_j$ read
\begin{equation}
\label{theorem refined spectral asymptotics equation 3}
a_{d-1}^{(j)}(x)=\frac{d}{(2\pi)^d} \int\limits_{h^{(j)}(x,\xi)<1}d\xi\,,
\end{equation}
\begin{multline}
\label{theorem refined spectral asymptotics equation 4}
a_{d-2}^{(j)}(x)=-\frac{d(d-1)}{(2\pi)^d}
\ \int\limits_{h^{(j)}(x,\xi)<1}
\operatorname{tr}\left(
P^{(j)}A_\mathrm{sub}
+\frac i2\{P^{(j)}, P^{(j)}\}A_\mathrm{prin}
\right.
\\
\left.
-\frac {1}{d-1}h^{(j)}(P_j)_\mathrm{sub}
\right)
(x,\xi)\,
d\xi\,,
\end{multline}
where
curly brackets denote the Poisson bracket
$\{B,C\}:=\sum_{\alpha=1}^d(B_{x^\alpha} C_{\xi_\alpha}- B_{\xi_\alpha} C_{x^\alpha})$
on matrix-functions on the cotangent bundle.
\end{enumerate}
\end{theorem}

\begin{proof}
(a) Formula \eqref{theorem refined spectral asymptotics equation 1} follows immediately from Proposition~\ref{proposition U(j) vs U_Aj} and \cite[Equation~(8.2)]{dirac}.

(b) Formula \eqref{theorem refined spectral asymptotics equation 2} is an immediate consequence of \cite[Equation~(8.2)]{dirac}, \eqref{U(t) as sum of U(j)j} and
\eqref{theorem refined spectral asymptotics equation 1}.

(c) Parts (a) and (b) imply that formulae \eqref{theorem refined spectral asymptotics equation 3} and \eqref{theorem refined spectral asymptotics equation 4} can be obtained from \cite[Formula~(1.23)]{CDV} and \cite[Formula~(1.24)]{CDV}, respectively, by dropping the summation over $j$. There is an additional factor $d$ in the RHS of \eqref{theorem refined spectral asymptotics equation 3} and an additional factor $(d-1)$ in the RHS \eqref{theorem refined spectral asymptotics equation 4}: this accounts for the somewhat nonstandard definition of Weyl coefficients adopted in this paper, compare \eqref{expansion for mollified derivative of counting function} and \cite[Formula~(1.6)]{CDV}.

Finally, in recasting \cite[Formula~(1.24)]{CDV} as \eqref{theorem refined spectral asymptotics equation 4} we used the identities
\begin{equation}
\label{proof theorem refined spectral asymptotics equation 1}
\{[v^{(j)}]^*,v^{(j)}\}=i\operatorname{tr}((P_j)_\mathrm{sub})
\end{equation}
and
\begin{equation}
\label{proof theorem refined spectral asymptotics equation 2}
\{
[v^{(j)}]^*,A_\mathrm{prin},v^{(j)} 
\}
=
-\operatorname{tr}(\{P^{(j)}, P^{(j)}\}A_\mathrm{prin})+h^{(j)}\{[v^{(j)}]^*,v^{(j)}\},
\end{equation}
where $v^{(j)}(x,\xi)$ denotes the normalised eigenvector of $A_\mathrm{prin}$ corresponding to the eigenvalue $h^{(j)}$ and
$\{ B,C,D\}:=\sum_{\alpha=1}^d(B_{x^\alpha} C D_{\xi_\alpha}- B_{\xi_\alpha} C D_{x^\alpha})$
is the generalised Poisson bracket. Formula \eqref{proof theorem refined spectral asymptotics equation 1} follows from \cite[Formula~(1.20)]{CDV} and Corollary~\ref{corollary initial condition vs projection} (see also \cite[Theorem~2.3]{part1}), whereas formula \eqref{proof theorem refined spectral asymptotics equation 2} is obtained via a lengthy but straightforward calculation involving \eqref{proof projections A vs root A equation 1} and properties of pseudodifferential projections.
\end{proof}

In plain English, Theorem~\ref{theorem refined spectral asymptotics} tells us that, when applying Levitan's wave method \cite{levitan} to the computation of spectral asymptotics for first order systems, the oscillatory integral $U^{(j)}(t)$ accounts for precisely the $j$-th of the $m$ sequences of eigenvalues into which the spectrum of $A$ was partitioned in Section~\ref{Spectral analysis: partitioning the spectrum}.

\section*{Acknowledgements}
\addcontentsline{toc}{section}{Acknowledgements}

We are grateful to Grigori Rozenbloum for useful bibliographic suggestions at an early stage of this work
and to David Edmunds for bringing the monograph \cite{schmudgen} to our attention.

MC was supported by a Leverhulme Trust Research Project Grant and by a Research Grant (Scheme 4) of the London Mathematical Society. Both are gratefully acknowledged.

\begin{appendices}

\section{Simultaneous diagonalization of unbounded operators}
\label{appendix on linear algebra}

In this appendix we present some results from functional analysis. The purpose is to provide motivation for
Theorems~\ref{Main results theorem 4}--\ref{Main results theorem 6}
in the main text of the paper. All operators in this appendix are assumed to be linear.

For the sake of clarity, let us start with the finite-dimensional setting.
Let $H$ be an $n$-dimensional complex inner product space.
Given a self-adjoint operator $A:H\to H$ and a number $\lambda\in\mathbb{R}$
we denote by $\Pi^+(A;\lambda)$ the orthogonal projection onto the eigenspaces of $A$
corresponding to eigenvalues greater than zero and less than $\lambda$.
We also employ the notation~\eqref{nonnegative  part of B}.

\begin{theorem}
\label{linear algebra theorem}
Let $A$ and $A_j$, $j=1,\dots,p$, be self-adjoint operators.
Suppose that
\begin{equation}
\label{linear algebra theorem equation 1}
A^+=\sum_{j=1}^pA_j^+.
\end{equation}
\begin{equation}
\label{linear algebra theorem equation 2}
A_j^+A_l^+=0,\qquad j,l=1,\dots,p,\quad j\ne l.
\end{equation}
Then
\begin{equation}
\label{linear algebra theorem equation 3}
\Pi^+(A;\lambda)
=
\sum_{j=1}^p
\Pi^+(A_j;\lambda).
\end{equation}
\end{theorem}

\begin{proof}
The self-adjoint operators $A_j^+$, $j=1,\dots,p$, commute,
hence one can choose a basis which simultaneously diagonalizes them
\cite[Theorem 2.3.3]{Horn}.
The diagonal entries in the matrix representations of the $A_j^+$s
are either zeros or positive numbers, and formula \eqref{linear algebra theorem equation 2}
tells us that for different $j$ the positive elements in the matrix representations of the $A_j^+$s
are in different positions. This immediately implies \eqref{linear algebra theorem equation 3}.
\end{proof}

Let us now proceed to the infinite-dimensional setting.
In what follows $H$ is a separable complex Hilbert space.
We will be dealing with self-adjoint operators which are not necessarily bounded
and this leads to a number of difficulties.
Indeed, generalising Theorem \ref{linear algebra theorem}
to infinite-dimensional spaces turns out to be a delicate matter.

Let us introduce the following definitions.

\begin{definition}
\label{definition of invariant subspace}
Let $A:D\to H$ be a self-adjoint operator and $V\subseteq D$ be a vector subspace.
We say that $V$ is an \emph{invariant subspace} of the operator $A$ if $A(V)\subseteq V$.
\end{definition}

\begin{definition}
\label{definition of invariant subspace proper 1}
We say that an invariant subspace $V$ of the self-adjoint operator $A$ is \emph{proper} 
if, for some $\lambda$ in the resolvent set $\rho(A)$, the map $A-\lambda\operatorname{Id}:V\to V$
is surjective, and, hence, bijective.
\end{definition}

The above definition can be equivalently recast as follows.

\begin{definition}
\label{definition of invariant subspace proper 2}
We say that an invariant subspace $V$ of the self-adjoint operator $A$ is \emph{proper} 
if, for some $\lambda\in\rho(A)$, $V$ is an invariant subspace of the resolvent $(A-\lambda\operatorname{Id})^{-1}$.
\end{definition}

\begin{example}
\label{example with proper and improper invariant subspaces}

\

\begin{enumerate}[(a)]
\item
Any finite-dimensional invariant subspace is proper.
\item
Let $H=L^2(M)$ and
let $A\in\Psi^s$, $s\in\mathbb{R}$, $s>0$, be an elliptic self-adjoint operator
(see Section \ref{Statement of the problem} for notation).
Then $C^\infty(M)$ is a proper invariant subspace.
\item
Let $H=l^2$, the Hilbert space of square summable sequences $x=(x_1,x_2,\dots)$.
Let $S:(x_1,x_2,\dots)\mapsto(0,x_1,x_2,\dots)$ be the right shift operator.
Put $A:=S+S^*$ and let $V=c_{00}$, the vector subspace of sequences that are
eventually zero. Then $V$ is an invariant subspace of the operator $A$ but it is not proper.
Indeed, take $y=(1,0,0,\dots)\in V$. It is easy to see that for any $\lambda\in\mathbb{C}$ there
does not exist an $x\in V$ satisfying $(A-\lambda\operatorname{Id})x=y$.
\end{enumerate}
\end{example}

The following theorem gives sufficient conditions for the simultaneous diagonalizability of
a family of unbounded commuting self-adjoint operators.

\begin{theorem}
\label{diagonalization theorem for family of unbounded operators}
Let $H$ be an infinite-dimensional separable complex Hilbert space and let $A_j$, $j=1,\dots,p$,
be self-adjoint operators
with discrete spectra,
which admit a common proper invariant subspace $V$ dense in~$H$.
If
\begin{equation}
\label{commutation on V}
[A_j,A_l]=0\quad\text{on}\quad V,\qquad j,l=1,\dots,p,
\end{equation}
then there exists an orthonormal basis $\{u_k\}$ such that
each basis element $u_k$  is an eigenvector of
$A_j$ for every $j=1,\dots,p$.
\end{theorem}

\begin{proof}
Consider the pair of operators $A_j$ and $A_l$ for some $j\ne l$.
Definitions~\ref{definition of invariant subspace}, \ref{definition of invariant subspace proper 1}
and formula \eqref{commutation on V} imply that there exist $\lambda\in\rho(A_j)$ and
$\mu\in\rho(A_l)$ such that the resolvents
$(A_j-\lambda\operatorname{Id})^{-1}$ and $(A_l-\mu\operatorname{Id})^{-1}$
commute on $V$.
Resolvents are bounded operators and $V$ is dense in $H$, hence
$[(A_j-\lambda\operatorname{Id})^{-1},(A_l-\mu\operatorname{Id})^{-1}]=0$ on $H$.
This, in turn, implies that the operators $A_j$ and $A_l$ \emph{strongly commute}
in the sense of \cite[Definition~5.2]{schmudgen} in view of \cite[Proposition~5.27]{schmudgen}.
The result now follows from \cite[Theorem~5.21]{schmudgen}.
\end{proof}

Arguing along the lines of the proof of Theorem~\ref{linear algebra theorem},
we see that Theorem~\ref{diagonalization theorem for family of unbounded operators}
immediately implies the following.

\begin{theorem}
\label{linear algebra theorem infinite-dimensional}
Let $H$ be an infinite-dimensional separable complex Hilbert space and let
$A$ and $A_j$, $j=1,\dots,p$,
be self-adjoint operators with the same domain $D$.
Suppose that the operators $A_j$, $j=1,\dots,p$,
have discrete spectra and admit a common proper invariant subspace $V$ dense in~$H$.
Furthermore suppose that conditions
\eqref{linear algebra theorem equation 1}
and
\eqref{commutation on V} are fulfilled as well as
\begin{equation}
\label{composition of A plus}
\langle A_j^+(\,\cdot\,),A_l^+(\,\cdot\,)\rangle=0\quad\text{on}\quad D\times D,\qquad j,l=1,\dots,p,\quad j\ne l.
\end{equation}
Then we have \eqref{linear algebra theorem equation 3}.
\end{theorem}

Of course, taking the trace in \eqref{linear algebra theorem equation 3} one obtains an
analogous result for the counting functions:
$N^+(A;\lambda)
=
\sum_{j=1}^p
N^+(A_j;\lambda)$.
Here by $N^+(\,\cdot\,;\lambda)$ we denote the number of eigenvalues, with account of multiplicity,
greater than zero and less than $\lambda$.

Examination of
Theorem~\ref{linear algebra theorem infinite-dimensional}
and
Example~\ref{example with proper and improper invariant subspaces}(b) leads to the following corollary.

\begin{corollary}
\label{main corollary in appendix}
Let $A$ and $A_j$, $j=1,\dots,p$, be elliptic self-adjoint operators
from the class $\Psi^s$, $s\in\mathbb{R}$, $s>0$.
Suppose that conditions
\eqref{linear algebra theorem equation 1}
and
\eqref{linear algebra theorem equation 2}
are fulfilled and that the  $A_j$, $j=1,\dots,p$, commute.
Then we have \eqref{linear algebra theorem equation 3}.
\end{corollary}

For the sake of clarity, let us point out that the proper invariant subspace underpinning the above
corollary is $C^\infty(M)$, because
\begin{itemize}
\item
elliptic self-adjoint pseudodifferential operators and their resolvents map $C^\infty(M)$ to $C^\infty(M)$ and
\item
pseudodifferential operators form an algebra.
\end{itemize}

\begin{remark}
\label{compelling remark}
Corollary~\ref{main corollary in appendix}
connects with the arguments presented in Section~\ref{Main results}
in that it provided strong motivation for our original conjecture on the structure of the
spectrum of the operator~$A$, compare formulae
\eqref{linear algebra theorem equation 1},
\eqref{linear algebra theorem equation 2}
and $[A_j,A_l]=0$
with
\eqref{nonnegative part of A equation 1},
\eqref{nonnegative part of A equation 2}
and
\eqref{Aj almost commute}.
Note, however, that in the main text the three conditions
are not satisfied precisely but only modulo $\Psi^{-\infty}$.
This calls for a more delicate spectral theoretic analysis
than that given in this appendix
and is ultimately responsible
for the appearance of remainders in our main results,
Theorems~\ref{Main results theorem 4}--\ref{Main results theorem 6}.
\end{remark}

\section{Weyl asymptotics for elliptic systems}
\label{Weyl asymptotics for elliptic systems}

In this appendix we provide, for the sake of completeness, a short proof of the Weyl law (one-term asymptotics with rough remainder estimate) for elliptic systems of arbitrary positive order. Though obtaining this result does not pose significant challenges, we were unable to find a rigorous proof for it in the literature.

Note that
\begin{enumerate}[(a)]
\itemsep0em
\item
we are dealing with a \emph{system} as opposed to a scalar operator,

\item
we allow the order of the operator to be \emph{any positive real number} and

\item
the operator is \emph{not necessarily semi-bounded.}

\end{enumerate}

Spectral theory for elliptic systems has a long and troubled history, see~\cite[Section~11]{CDV} for a review. Two-term asymptotic formulae for the counting function of a first order system ($s=1$) were recently obtained by Chervova, Downes and
Vassiliev \cite{CDV}, see also \cite{AFV,ASV}. Some results are available for particular special cases, e.g.~two-term asymptotics for nonnegative (pseudo)differential operators of even order are given in \cite{dima_old,safarov}, but we are unaware of general results.

Let $A\in\Psi^s$, $s>0$, be an operator as in Section~\ref{Statement of the problem}
and let
\begin{equation}
\label{Weyl asymptotics for elliptic systems counting function definition}
N^+(\lambda):=
\begin{cases}
0 & \text{for}\ \lambda\le 0,\\
\sum_{k\,:\,0<\lambda_k<\lambda}1 &\text{for}\ \lambda> 0\\
\end{cases}
\end{equation}
be its positive counting function.

\begin{theorem}
\label{Weyl asymptotics for elliptic systems theorem}
We have
\begin{equation}
\label{Weyl asymptotics for elliptic systems theorem equation 1}
N^+(\lambda)=
b\lambda^{d/s}+o(\lambda^{d/s})
\quad\text{as}\quad\lambda\to+\infty,
\end{equation}
where
\begin{equation}
\label{Weyl asymptotics for elliptic systems theorem equation 2}
b=
\frac{1}{(2\pi)^d}\,\sum_{j=1}^{m^+}
\ \int\limits_{h^{(j)}(x,\xi)<1}\operatorname{dVol}_{T^*M}.
\end{equation}
\end{theorem}

\begin{proof}
Consider the function
\begin{equation*}
\label{Weyl asymptotics for elliptic systems theorem equation 1 proof}
f:(0,+\infty)\to[0,+\infty),
\qquad
f(t):=
\operatorname{Tr}
\bigl(
\theta(A)\,e^{-t|A|}
\bigr)
=\int_{-\infty}^{+\infty}e^{-t\lambda}\,dN^+(\lambda)\,.
\end{equation*}
The operators $|A|$ and $\theta(A)$ are  pseudodifferential operators
and this allows us to apply to the operator $\,\theta(A)\,e^{-t|A|}\,$ the standard technique from
\cite{seeley}, giving us the asymptotic formula
\begin{equation}
\label{Weyl asymptotics for elliptic systems theorem equation 2 proof}
f(t)=
b\,\Gamma\left(\frac ds+1\right)
t^{-d/s}+o(t^{-d/s})\quad\text{as}\quad t\to 0^+,
\end{equation}
where $\Gamma$ is the Gamma function.
Karamata's Tauberian theorem \cite[Problem 14.2]{shubin}
tells us that
\eqref{Weyl asymptotics for elliptic systems theorem equation 2 proof}
implies
\eqref{Weyl asymptotics for elliptic systems theorem equation 1}.
\end{proof}

\begin{remark}
Of course, for $s=1$ the coefficient $b$ appearing in Theorem~\ref{Weyl asymptotics for elliptic systems theorem}
is related to the coefficient $a_{d-1}(x)$ appearing in
formula~\eqref{expansion for mollified derivative of counting function}
as
\begin{equation*}
\label{Weyl asymptotics for elliptic systems theorem relation between a and b}
b=
\frac{1}d
\ \int\limits_M a_{d-1}(x)\,dx\,.
\end{equation*}
\end{remark}

\end{appendices}

\end{document}